\documentclass[12pt]{amsart}

\usepackage{amssymb, amsmath, amsthm}
\usepackage[margin=1in]{geometry}
\usepackage[dvipdfmx]{graphicx,xcolor} 
\usepackage{tikz}

\allowdisplaybreaks

\numberwithin{equation}{section}

\theoremstyle{plain}
\newtheorem{thm}{Theorem}[section]
\newtheorem{prop}[thm]{Proposition}
\newtheorem{lem}[thm]{Lemma}

\newtheorem*{referthmA}{Theorem A}
\newtheorem*{referthmB}{Theorem B}

\theoremstyle{definition}
\newtheorem{defn}[thm]{Definition}
\newtheorem{rem}[thm]{Remark}

\newtheorem{notation}[thm]{Notation}

\newcommand{\ichi}{\mathbf{1}}

\newcommand{\N}{\mathbb{N}}
\newcommand{\R}{\mathbb{R}}

\newcommand{\Z}{\mathbb{Z}}

\newcommand{\calM}{\mathcal{M}}

\newcommand{\calS}{\mathcal{S}}

\newcommand{\w}{\mathrm{w}}
\newcommand{\supp}{\mathrm{supp}\, }

\newcommand{\card}{\mathrm{card}\, }

\newcommand{\grad}{\, \mathrm{grad}\, }
\newcommand{\RealPart}{\, \mathrm{Re}\, }

\newcommand{\RomI}{\mathrm{I}}
\newcommand{\II}{\mathrm{I}\hspace{-0.5pt}\mathrm{I}}
\newcommand{\III}{\mathrm{I}\hspace{-0.5pt}\mathrm{I}\hspace{-0.5pt}\mathrm{I}}

\begin{document}

\title[Bilinear Fourier multipliers]
{On some bilinear Fourier multipliers with 
oscillating factors, I} 

\author[T. Kato]{Tomoya Kato}
\author[A. Miyachi]{Akihiko Miyachi}
\author[N. Shida]{Naoto Shida}
\author[N. Tomita]{Naohito Tomita}

\address[T. Kato]
{Faculty of Engineering, 
Gifu University, 
Gifu 501-1193, Japan}

\address[A. Miyachi]
{Department of Mathematics, 
Tokyo Woman's Christian University, 
Zempukuji, Suginami-ku, Tokyo 167-8585, Japan}

\address[N. Shida]
{Graduate School of Mathematics, 
Nagoya University, 
Furocho, Chikusa-ku, Nagoya 464-8602, Japan}

\address[N. Tomita]
{Department of Mathematics, 
Graduate School of Science, Osaka University, 
Toyonaka, Osaka 560-0043, Japan}

\email[T. Kato]{kato.tomoya.s3@f.gifu-u.ac.jp}
\email[A. Miyachi]{miyachi@lab.twcu.ac.jp}
\email[N. Shida]{naoto.shida.c3@math.nagoya-u.ac.jp}
\email[N. Tomita]{tomita@math.sci.osaka-u.ac.jp}

\date{\today}

\thanks{This work was supported by JSPS KAKENHI, 
Grant Numbers 
23K12995 (Kato), 
23K20223 (Miyachi), 
23KJ1053 (Shida), 
and 
20K03700 (Tomita).}
%
%

\maketitle

\begin{abstract}
Bilinear Fourier multipliers of the form 
$e^{i (|\xi| + |\eta|+ |\xi + \eta|)} \sigma (\xi, \eta)$ 
are considered. 
It is proved that if $\sigma (\xi, \eta)$ is in the 
H\"ormander class $S^{m}_{1,0} (\R^{2n})$ with 
$m=-(n+1)/2$ then the corresponding 
bilinear operator is bounded in 
$L^{\infty}\times L^{\infty} \to bmo$, 
$h^{1} \times L^{\infty} \to L^{1}$, and 
$L^{\infty} \times h^{1} \to L^{1}$.  
This improves a result given by 
Rodr\'iguez-L\'opez, Rule and Staubach. 
\end{abstract}

\section{Introduction}
\label{intro}

Throughout this paper, 
the letter $n$ denotes the dimension of the Euclidean space 
that we consider. 
Unless further restrictions are explicitly made, 
$n$ denotes an arbitrary positive integer. 
For other notation, see Notation \ref{notation} given at  
the end of this section.

We first recall a result for linear Fourier multipliers.  
Let $\theta =\theta (\xi)$ be a measurable function on $\R^n$ 
that is locally integrable and of at most polynomial 
growth as $|\xi| \to \infty$. 
Then the operator $\theta (D)$ is defined by 
\[
\theta (D) f (x) 
=
\frac{1}{(2\pi)^{n}}
\int_{\R^n}
e^{i x \cdot \xi}\, 
\theta (\xi) 
\widehat{f}(\xi) 
\, d\xi, 
\quad 
x \in \R^n, 
\]
for $f$ in the Schwartz class $\calS = \calS (\R^n)$, 
where $\widehat{f}$ denotes the Fourier transform. 
The function $\theta$ is called the multiplier.  
If $(X, \|\cdot \|_X)$ and $(Y, \|\cdot \|_{Y})$ 
are function spaces on $\R^n$, 
and if there exists a constant $A$ 
such that 
\[
\| 
\theta (D) f \|_{Y} 
\le A \|  f \|_{X} 
\quad \text{for all}\;\; 
f \in \calS \cap X, 
 \]
then we say that 
$\theta$ is a 
{\it Fourier multiplier}\/  
for $X\to Y$ 
and write $\theta \in \calM (X\to Y)$. 
(Sometimes we write 
$\theta (\xi) \in \calM (X\to Y)$ to mean 
$\theta (\cdot) \in \calM (X\to Y)$.)  
The minimum of the constant $A$ 
is denoted by $\|\theta\|_{\calM (X\to Y)}$.

The following class is frequently used in the theory of Fourier multipliers.

\begin{defn}\label{def-Sm10} 
For $d\in \N$ and $m\in \R$, 
the class $S^{m}_{1,0} (\R^{d})$ is defined to be the set of all 
$C^{\infty}$ functions $\sigma = \sigma (\zeta)$ 
on $\R^{d}$ 
that satisfy the estimate 
\begin{equation*}
|\partial_{\zeta}^{\alpha}
\sigma (\zeta)| 
\le 
C_{\alpha} \big( 1+ |\zeta| \big)^{m-|\alpha|}
\end{equation*}
for all multi-indices $\alpha$. 
\end{defn}

The following theorem is known. 

\begin{referthmA}[Seeger--Sogge--Stein] 
Let $\phi$ be a real-valued smooth 
positively homogeneous function of degree one on 
$\R^n \setminus \{0\}$.   
Let $1\le p\le \infty$. 
Then for every 
$\sigma \in S^{m}_{1,0}(\R^n)$ with 
$m= -(n-1)  | {1}/{p} - {1}/{2} |$ 
the function 
$e^{i \phi (\xi)} \sigma (\xi)$ 
is a Fourier multiplier for 
$h^{p} \to L^{p} $,
where the target space $L^p$ should be replaced by $bmo$ 
when $p=\infty$. 
\end{referthmA}

This theorem is due to Seeger--Sogge--Stein \cite{SSS}, 
where a more general operator (Fourier integral operator) 
is treated. 
The special case $\phi (\xi)=|\xi|$ is treated by 
Peral \cite{Peral} and Miyachi \cite{M-wave}. 
It is known that the number 
$ -(n-1)  | {1}/{p} - {1}/{2} | $ of this theorem 
is sharp; see S. Sj\"ostrand \cite[Lemma 2.1 and Theorem 5.2]{Sj}. 
Proof of the sharpness of the number 
$ -(n-1)  | {1}/{p} - {1}/{2} | $ 
can also be found in 
\cite[Chapter IX, Section 6.13 (c)]{St}.

We shall be interested in bilinear versions of Theorem A.

We recall the definition of bilinear Fourier multiplier operators. 
Let $\sigma = \sigma (\xi, \eta)$ be a measurable 
function on $\R^{2n}$ 
that is locally integrable and of at most polynomial 
growth as $|\xi| +|\eta| \to \infty$. 
Then the bilinear operator
$T_{\sigma}$ is defined by
\[
T_{\sigma}(f, g)(x)
=\frac{1}{(2\pi)^{2n}}
\iint_{\R^n\times \R^n} 
e^{i x \cdot(\xi + \eta)}
\sigma (\xi, \eta) \,
\widehat{f}(\xi)\, 
\widehat{g}(\eta)
\, d\xi d\eta, 
\quad 
x \in \R^n, 
\]
for $f, g \in \calS$. 
If $X, Y$, and $Z$ are 
function spaces on $\R^n$ 
equipped with quasi-norms or seminorms 
$\|\cdot \|_{X}$, $\|\cdot \|_{Y}$, and $\|\cdot \|_{Z}$, 
respectively, and if there exists a constant $A$ such that 
\begin{equation*}
\|T_{\sigma}(f, g)\|_{Z}
\le A \|f\|_{X}\, \|g\|_{Y} 
\quad 
\text{for all}
\;\;
f \in \calS \cap X 
\;\; 
\text{and}\;\; 
g \in \calS \cap Y, 
\end{equation*}
then 
we say that 
$\sigma$ is a {\it bilinear Fourier multiplier}\/ for 
$X \times Y$ to $Z$ and 
write 
$\sigma \in \calM (X \times Y \to Z)$.  
(Sometimes we write 
$\theta (\xi, \eta) \in \calM (X\times Y \to Z)$ to mean 
$\theta (\cdot, \cdot) \in \calM (X\times Y \to Z)$.)  
The minimum of the constant $A$ 
is denoted by
$\|\sigma\|_{\calM (X\times Y \to Z)}$.

The work of 
Grafakos--Peloso \cite{Grafakos-Peloso} 
seems to be the first one that considered  
a bilinear version of Theorem A. 
Among general bilinear Fourier integral operators, 
they considered a bilinear multiplier of the form 
\begin{equation}\label{eq-2phase}
e^{i (\phi_1 (\xi) + \phi_2 (\eta))} \sigma (\xi, \eta), 
\quad \sigma \in S^{m}_{1,0}(\R^{2n}), 
\end{equation}
where $\phi_1$ and $\phi_2$ are real-valued 
smooth positively homogeneous functions of degree one 
on $\R^n \setminus \{0\}$.  
They proved that 
\eqref{eq-2phase} is a 
bilinear Fourier multiplier for 
$h^{p}\times h^{q} \to L^{r}$ 
if $1\le p, q<2$, $1/p + 1/q = 1/r$, 
and $m=-(n-1)(1/p+1/q-1)$; 
see \cite[Proposition 2.4 and Remark 6.2]{Grafakos-Peloso}. 
(To be precise, 
the result of 
\cite[loc. cit.]{Grafakos-Peloso} 
is restricted to local estimates.)

Rodr\'iguez-L\'opez--Rule--Staubach \cite{RRS1} 
also considered  bilinear 
Fourier multipliers of the form 
\eqref{eq-2phase}. 
Among several results concerning Fourier integral operators, 
they extended the 
result of \cite{Grafakos-Peloso} by proving 
that \eqref{eq-2phase} is 
a bilinear Fourier multiplier for 
$H^{p}\times H^{q} \to L^r$ 
if $n\ge 2$, $1\le p, q\le \infty$, $1/p + 1/q = 1/r$, 
and $m=-(n-1)(|1/p-1/2|+|1/q-1/2|)$, 
where $L^r$ should be replaced by $BMO$ when $r=\infty$; 
see \cite[Theorem 2.7]{RRS1}. 
(To be precise, the paper \cite{RRS1}  
also concerns with Fourier integral operators and 
the given estimate is a local one; 
extension to global estimates including the multilinear case 
is given in the recent paper 
Bergfeldt--Rodr\'iguez-L\'opez--Rule--Staubach 
\cite[Theorem 1.4 and Remark 1.5]{BRRS-Trans}.)

Kato--Miyachi--Tomita \cite{KMT-RMI, KMT-JPOA} 
also considered bilinear Fourier multipliers of the 
form \eqref{eq-2phase} and 
proved that the condition $m=-(n-1)(|1/p-1/2|+|1/q-1/2|)$ 
of \cite{RRS1} 
can be relaxed in the case $p<2<q$ or $q<2<p$. 
In \cite{KMT-RMI}, it is shown that 
the condition $m=-(n-1)(|1/p-1/2|+|1/q-1/2|)$ 
is sharp in the range 
$1\le p,q\le 2$ or $2\le p, q \le \infty$, 
but the sharp condition 
for the case  $1/p+1/q=1$ is $m=-n|1/p-1/2|$.

Recently, 
Rodr\'iguez-L\'opez--Rule--Staubach \cite{RRS2} 
considered another bilinear version of Theorem A. 
They considered the bilinear multiplier of the form 
\begin{equation}\label{eq-3phase}
e^{i (\phi_1 (\xi) + \phi_2 (\eta) + 
\phi_3 (\xi+\eta))} \sigma (\xi, \eta), 
\quad \sigma \in S^{m}_{1,0}(\R^{2n}), 
\end{equation}
where $\phi_1$, $\phi_2$, and $\phi_3$ 
are real-valued smooth 
positively homogeneous functions of degree one 
on $\R^n \setminus \{0\}$. 
In fact \cite{RRS2} considers 
more general operators, the 
multilinear Fourier integral operators, 
but here we shall refer only to their result 
that concerns with the 
above simple Fourier multipliers. 
In \cite[Theorem 1.4]{RRS2}, 
the following theorem is proved.

\begin{referthmB}[{Rodr\'iguez-L\'opez--Rule--Staubach}]
Let $n\ge 2$. 
Let $\phi_{1}$, $\phi_{2}$, and $\phi_{3}$ be 
real-valued smooth 
positively homogeneous functions of degree one on 
$\R^n \setminus \{0\}$. 
Suppose $1\le p, q, r\le \infty$ and $1/p+1/q=1/r$. 
Define $m_{1}(p,q)\in \R$ by 
\begin{equation*}
m_{1}(p,q)
= 
-(n-1) \big( | {1}/{p} - {1}/{2} | 
+ | {1}/{q} - {1}/{2} | +|1/r-1/2|\big). 
\end{equation*}
Then 
for every $\sigma \in S^{m}_{1,0}(\R^{2n})$ 
with $m=m_{1}(p,q)$ the function  
\eqref{eq-3phase} is a bilinear Fourier multiplier for 
$h^{q} \times h^{q} \to L^{r} $,
where $L^r$ should be replaced by $bmo$ when $r=\infty$. 
\end{referthmB}

Now the main purpose of the present paper is to show that, 
in the special case where 
$\phi_{1}=\phi_{2}=\phi_{3} =|\cdot|$, 
the condition $m=m_1 (p,q)$ of Theorem B 
for 
$(p,q)=(\infty, \infty), (1, \infty), (\infty, 1)$ 
can be relaxed.

The following is the main result of the present paper. 

\begin{thm}\label{thm_Main}
Let $n\ge 2$. 
Then for every 
$\sigma \in S^{m}_{1,0}(\R^{2n})$ 
with $m =-(n+1)/2$, 
the function 
$e^{i (|\xi| + |\eta|+ |\xi + \eta|)} \sigma (\xi, \eta)$ 
is a bilinear Fourier multiplier for 
$L^{\infty} \times L^{\infty} \to bmo$, 
$h^{1} \times L^{\infty} \to L^{1}$, and 
$L^{\infty} \times h^{1} \to L^{1}$.  
\end{thm}

Notice that 
$m_1 (p,q)$ of Theorem B gives 
$m_1 (\infty, \infty)=m_1 (1, \infty)=m_1 (\infty, 1)=-3(n-1)/2$, 
which is smaller than the number $-(n+1)/2$ of 
Theorem \ref{thm_Main} 
in the case $n\ge 3$.

As for the condition $m=-(n+1)/2$ 
of Theorem \ref{thm_Main}, 
the present authors do not know whether it is 
sharp. 
However, concerning this condition, 
we shall prove a weak result that 
a key inequality in our proof of Theorem \ref{thm_Main} cannot 
be extended to the case $m>-(n+1)/2$. 
The precise statement will be given in 
Proposition \ref{prop_Main2} 
in Section \ref{necessary_conditions_m}. 
Notice that it does not prove the sharpness of the 
number $-(n+1)/2$ of Theorem \ref{thm_Main}.  
For the condition on $m$ in \eqref{eq-3phase},  
see also Remark \ref{rem_introduction} (1) below.

Here are some remarks concerning 
the subject of this paper.

\begin{rem}\label{rem_introduction} 
(1) 
It should be an interesting problem to 
find sharp condition on $m$ 
for which the multiplier \eqref{eq-3phase} 
be a bilinear Fourier multiplier 
for $h^p \times h^q \to L^r$ or $bmo$. 
Observe that $m_{1}(p,q)$ of Theorem B is equal to 
$-(n-1)/2$ if  
$2\le p, q, r\le \infty$. 
In \cite[Section 4, p.~17]{RRS2}, 
it is shown 
that $m=-(n-1)/2$ is the sharp condition 
in the case $p=q=4$ and $r=2$. 
In \cite[the end of Section 3, p.~15]{BRRS-Trans}, 
it is shown that $m=-(n-1)/2$ is also sharp in the case 
$2\le p, q\le \infty$ and $r=2$. 
If we combine these results with a simple 
argument using interpolation and duality, 
we see that  
for every $(p,q, r)$ satisfying $1\le p, q, r\le \infty$ 
and $1/p+1/q=1/r$ 
the condition $m\le -(n-1)/2$ is necessary 
for all the multipliers  
\eqref{eq-3phase} to be 
bilinear Fourier multipliers  
for $h^p \times h^q \to L^r$, where 
$L^r$ should be replaced by $BMO$ 
when $r=\infty$.

(2) 
It may be an interesting problem to generalize 
Theorem \ref{thm_Main} to the multipliers 
of the form \eqref{eq-3phase} with 
$\phi_{\ell}$ being general homogeneous functions 
of degree one. 
Our proof of Theorem \ref{thm_Main} relies on 
the estimate of the function 
$\big(e^{i|\xi|} \theta (2^{-j}\xi)\big)^{\vee}(x)$ 
for $\theta \in C_{0}^{\infty}(\R^n)$, 
in particular the estimate of 
the $L^{\infty}$-norm 
of this function is crucial in our proof. 
But the $L^{\infty}$-estimate 
cannot be extended to 
$\big(e^{i \phi (\xi)} \theta (2^{-j}\xi)\big)^{\vee}(x)$  
with $\phi$ being a 
general positively homogeneous function on degree one 
(see Remark \ref{rem_impossiblecase} (2)). 
\end{rem}

In the forthcoming article \cite{KMST2}, 
the same authors will consider bilinear Fourier multipliers 
of the form 
\begin{equation*}
e^{i (|\xi|^{s} + |\eta|^{s} + 
|\xi+\eta|^{s})} \sigma (\xi, \eta), 
\quad s>0, \;\;  s\neq 1.  
\end{equation*}
It uses several ideas used in the present paper, 
but the results are independent.

The contents of the rest of the paper are as follows. 
In Section \ref{preliminaries}, 
we shall give some preliminary argument for 
the multipliers of the form \eqref{eq-3phase}. 
In Section \ref{kernel}, 
we give estimates 
for functions of the form 
$\big(e^{i|\xi|} \theta (2^{-j}\xi)\big)^{\vee}(x)$, 
which will be used in the proof of Theorem \ref{thm_Main}.
In Section \ref{proofThmMain}, we prove Theorem \ref{thm_Main}. 
In Section \ref{necessary_conditions_m}, the last section, 
we give the proposition concerning the condition  
$m=-(n+1)/2$ as mentioned above.

We end this section by introducing some notations 
used throughout the paper.

\begin{notation}\label{notation} 
The letter $\N$ denotes the set of positive integers. 
For $E, F \subset \R^n$ and $v \in \R^n$, 
we write $E+F = \{x+y \mid x\in E, \; y\in F\}$ 
and  $v+F = \{v\}+F$.  
We write the ball of $\R^n$ as 
$B(x,r)= \{y\in \R^n \mid |x-y|<r\}$. 
A subset of $\R^n$ of the form 
$\{x\in \R^n \mid a<|x|<b\}$ with 
$0<a<b<\infty$ will be called an annulus. 

The Fourier transform and the inverse Fourier transform 
on $\R^d$ are defined by 
\begin{equation*}
\widehat{f}(\xi)
=
\int_{\R^d}
e^{-i \xi \cdot x}
f(x)\, dx, 
\quad 
(g)^{\vee}(x)
=
\frac{1}{(2\pi)^d} 
\int_{\R^d}
e^{i \xi \cdot x}
g(\xi)\, d\xi. 
\end{equation*}
Sometimes we use rude expressions 
$\big( f(x) \big)^{\wedge}$ 
or 
$\big( g(\xi) \big)^{\vee}$ 
to denote 
$\big( f(\cdot ) \big)^{\wedge}$ 
or  
$\big( g(\cdot ) \big)^{\vee}$, respectively.

We use the usual 
dyadic partition of unity, which is defined as follows.   
We take a function  
$\psi \in C_{0}^{\infty}(\R^n)$ such that 
$\supp \psi \subset \{2^{-1}\le |\xi| \le 2\}$ and 
$\sum_{j=-\infty}^{\infty} \psi (2^{-j} \xi)=1$ 
for all 
$\xi \neq 0$.
We define functions $\zeta$ and 
$\varphi$ by 
$\zeta (\xi)
=
\sum_{j=1}^{\infty} \psi (2^{-j}\xi)$ 
and 
$\varphi (\xi)= 1- \zeta (\xi)$. 
Thus $\zeta (\xi)=0$ for $|\xi|\le 1$,  
$\zeta (\xi)=1$ for $|\xi|\ge 2$, 
$\varphi (\xi)=1$ for $|\xi|\le 1$, and 
$\varphi (\xi)=0$ for $|\xi|\ge 2$. 
For $j\in \N \cup \{0\}$, we write 
\[
\psi_{j} (\xi)= 
\begin{cases}
{\psi (2^{-j}\xi)} & \text{if}\;\; {j\in \N}, \\
{\varphi (\xi)} & \text{if}\;\; {j=0}. 
\end{cases}
\]
Observe the following equality: 
\begin{equation}\label{eq-varphi}
\sum_{j=0}^{k} \psi_{j} (\xi) = \varphi (2^{-k} \xi),  
\quad k\in \N \cup \{0\}. 
\end{equation}
Notice, however, that we sometimes use the letter 
$\psi$ to denote other functions.

For a smooth function $\theta$ on $\R^d$ and for a nonnegative integer   
$M$, we write 
$\|\theta\|_{C^M}=
\max_{|\alpha|\le M} 
\sup_{\xi \in \R^d} \big| \partial_{\xi}^{\alpha} \theta (\xi) \big|$. 
For $0<p\le \infty$, $H^p$ denotes the usual 
Hardy space 
and $h^p$ denotes the local Hardy space. 
The space of bounded mean oscillation is denoted by 
$BMO$, which is the dual of $H^1$. 
The local version of $BMO$ is denoted by 
$bmo$, which is the dual space of $h^1$. 
We use the convention that $H^p =h^p= L^p$ if $1<p\le \infty$. 
For $H^p$ and $BMO$, see Stein \cite[Chapters III and IV]{St}. 
For $h^p$ and $bmo$, see 
Goldberg \cite{Goldberg}.  
\end{notation}

\section{Preliminaries}
\label{preliminaries}

In this section, we shall give a preliminary argument 
to reduce the claim 
\begin{equation}\label{eq201}
e^{i(|\xi|+|\eta|+|\xi+\eta|)} \sigma (\xi, \eta) 
\in \calM (X\times Y \to Z)
\quad 
\text{for all}
\quad 
\sigma \in S^{m}_{1,0}(\R^{2n}) 
\end{equation}
to simple inequalities. 
We shall also consider the necessary condition for 
the above claim. 
In this section, $X$, $Y$, and $Z$ denote 
general function spaces on $\R^n$.

\subsection{Reduction of the proof of the claim 
\eqref{eq201}}
\label{CMargument}

In this subsection, 
we shall recall the argument to reduce 
the proof of the claim \eqref{eq201} 
to the case where $\sigma$ is a sum of 
functions of product form. 
In fact, this argument is well-known. 
The idea goes back to 
Coifman and Meyer 
\cite[Chapter II, Lemma 4, p. 46; Chapter II, Section 13, 
pp. 54--58]{CM1}; 
a clear presentation can also be found in 
\cite[Section 7.5.2]{G-modern}. 
We shall briefly recall the argument.

Suppose a function $\sigma \in S^{m}_{1,0}(\R^{2n})$ is given.

Using the functions $\psi_j$ of Notation \ref{notation}, 
we decompose $\sigma$ as 
\begin{align*}
\sigma (\xi, \eta) 
&
=
\sum_{j, k} 
\sigma (\xi, \eta) \psi_j (\xi) \psi_k (\eta)
\\
&= 
\sum_{j-k \ge 3}
 + \sum_{j-k\le -3} + \sum_{-2\le j-k \le 2}
\\
&=
\sigma_{\RomI}(\xi, \eta) 
+ \sigma_{\II}(\xi, \eta) + \sigma_{\III}(\xi, \eta), 
\end{align*}
where the indices $j$ and $k$ run over nonnegative integers.

Consider $\sigma_{\RomI}$.  
By \eqref{eq-varphi}, this is written as 
\[
\sigma_{\RomI}(\xi, \eta) 
=
\sum_{j=3}^{\infty} \sum_{k=0}^{j-3} 
\sigma (\xi, \eta) \psi_j (\xi) \psi_k (\eta)
=
\sum_{j=3}^{\infty} 
\sigma (\xi, \eta) \psi (2^{-j}\xi) \varphi (2^{-j+3}\eta). 
\]
The assumption $\sigma \in S^{m}_{1,0}(\R^{2n})$ gives the 
estimate 
\[
\big| 
\partial_{\xi}^{\alpha}
\partial_{\eta}^{\beta} 
\sigma (\xi, \eta) 
\big| 
\le C_{\alpha, \beta} 2^{j(m-|\alpha|-|\beta|)}
\]
in a neighborhood of the support of the 
function $\psi (2^{-j}\xi) \varphi (2^{-j+3}\eta)$. 
Hence using the Fourier series expansion of period 
$2\pi \cdot 2^{j} \Z^{n} \times 2\pi \cdot 2^{j-3} \Z^n$ 
we see that 
$\sigma_{\RomI}$ is written as 
\begin{equation*}
\sigma_{\RomI}(\xi, \eta) 
=
\sum_{j=3}^{\infty} 
\sum_{a,b \in \Z^n}\, 
c_{\RomI, j} (a,b) 
e^{ia\cdot 2^{-j}\xi} 
e^{ib\cdot 2^{-j+3}\eta}
\psi (2^{-j} \xi)
\varphi (2^{-j+3} \eta)  
\end{equation*}
with the coefficient $c_{\RomI, j} (a,b) $ satisfying 
the estimate 
\begin{equation}\label{eq212}
\big| c_{\RomI, j} (a,b) \big| 
\lesssim 
2^{jm} (1+|a|)^{-L} (1+|b|)^{-L}
\end{equation}
for any $L>0$.

Similarly, $\sigma_{\II}$ and $\sigma_{\III}$ can be written as 
\begin{align*}
&\sigma_{\II}(\xi, \eta) 
=
\sum_{j=3}^{\infty} 
\sum_{a,b \in \Z^n}\, 
c_{\II, j} (a,b) 
e^{ia\cdot 2^{-j+3}\xi} 
e^{ib\cdot 2^{-j}\eta}
\varphi (2^{-j+3} \xi)
\psi (2^{-j} \eta),
\\
& 
\sigma_{\III}(\xi, \eta) 
=
\sum_{\ell = -2}^{2}\, 
\sum_{j=\max \{0, \ell\}}^{\infty} \, 
\sum_{a,b \in \Z^n}\, 
c_{\III, j, \ell} (a,b) 
e^{ia\cdot 2^{-j}\xi} 
e^{ib\cdot 2^{-j+\ell}\eta}
\psi (2^{-j} \xi)
\psi (2^{-j+\ell} \eta), 
\end{align*}
where 
$\psi (2^{-j} \xi)$ and $\psi (2^{-j+\ell} \eta)$ 
in $\sigma_{\III}(\xi, \eta) $ 
should be replaced by 
$\varphi (\xi)$ or $\varphi (\eta)$ 
if $j=0$ or $j-\ell =0$, respectively, 
and the coefficients $c_{\II, j} (a,b) $ 
and $c_{\III, j, \ell} (a,b) $ satisfy the same estimates as 
\eqref{eq212}.

Now, in view of the above decomposition 
of the functions of $S^{m}_{1,0}(\R^{2n})$, 
we see that 
in order to prove the claim \eqref{eq201} 
it is sufficient to show the estimate 
\begin{equation}\label{eq3000}
\big\| 
e^{i(|\xi|+|\eta|+|\xi+\eta|)} \sigma (\xi, \eta) 
\big\|_{\calM (X\times Y \to Z)}
\le c 
\|\theta_1\|_{C^M}
\|\theta_2\|_{C^M}, 
\end{equation}
with $c=c(n,m,X,Y, Z, \supp \theta_1, \supp \theta_2)$ 
and $M=M(n,m,X,Y, Z)$,   
in the following two cases. 
Case 1: $\theta_1$ and $\theta_2$ are functions in 
$C_{0}^{\infty}(\R^n)$ 
for which at least two of the sets 
$\supp \theta_1$, 
$\supp \theta_2$, and 
$\supp \theta_1 + \supp \theta_2$ 
are included in an annulus 
and $\sigma$ is defined by 
\begin{equation}\label{eq_sigma_case1}
\sigma (\xi, \eta) 
= 
\sum_{j=1}^{\infty} 
c_j 
\theta_1 (2^{-j} \xi) 
\theta_2 (2^{-j} \eta) 
\end{equation}
with complex numbers $c_j$ 
satisfying 
$|c_j| \le 2^{jm}$.  
Case 2:     
$\theta_1, \theta_2 
\in C_{0}^{\infty}(\R^n)$ and 
$\sigma (\xi, \eta) = \theta_1 (\xi) 
\theta_2 (\eta)$.

In fact, if \eqref{eq3000} is proved for Case 1, 
then, 
applying it to the case of functions 
$\theta_{1}(\xi)=e^{i a\cdot \xi} \psi (\xi)$ 
and $\theta_{2}(\eta)=e^{i b\cdot 2^3 \eta} \varphi (2^3 \eta)$, 
and to the coefficients 
$c_j = 
c_{\RomI, j}(a,b) 
(1+|a|)^{L} (1+|b|)^{L}$,  
we see that 
\begin{align*}
&
\bigg\| 
e^{i(|\xi|+|\eta|+|\xi+\eta|) } 
\sum_{j=3}^{\infty} 
c_{\RomI, j}(a,b) 
e^{ia\cdot 2^{-j}\xi} 
e^{ib\cdot 2^{-j+3}\eta}
\psi (2^{-j}\xi) \varphi (2^{-j+3}\eta)
\bigg\|_{\calM (X\times Y \to Z)}
\\
& 
\lesssim 
(1+|a|)^{-L+M} 
(1+|b|)^{-L+M}
\end{align*} 
and hence taking $L$ sufficiently large and taking 
sum over $(a,b)\in \Z^{2n}$, 
we see that 
$e^{i(|\xi|+|\eta|+|\xi+\eta|)} 
\sigma_{\RomI} (\xi, \eta) $ is 
a bilinear Fourier multiplier for $X\times Y \to Z$. 
Similar argument applies to 
$\sigma_{\II}$ and to the 
sum 
of the terms of $\sigma_{\III}$ 
corresponding to $j\ge 3$. 
The sum of the 
terms of $\sigma_{\III}$ 
corresponding to $j\le 2$ can be handled by Case 2.

\subsection{Use of duality} 
\label{Duality}

We shall observe the dual form of the inequality \eqref{eq3000}. 
 
Suppose we have a function space 
$(Z^{\prime}, \|\cdot \|_{Z^{\prime}})$ 
such that 
the duality relation 
\begin{equation}\label{eq218}
\big\| F \big\|_{Z}
\approx 
\sup\, \bigg\{ 
\bigg|
\int F(x) h(x)\, dx 
\bigg| 
\; \bigg| \; 
h \in \calS, \; \|h\|_{Z^{\prime}}\le 1
\bigg\}
\end{equation}
holds at least for all bounded continuous functions $F$ on $\R^n$.

Suppose 
$\theta_1, \theta_2, \theta_3\in C_{0}^{\infty}(\R^n)$ 
satisfy the condition 
\begin{equation}\label{eq_theta123}
\theta_3 (-\zeta) =1 
\quad 
\text{for all}
\quad 
\zeta \in \supp \theta_1 + \supp \theta_2. 
\end{equation}

Observe that if $\theta_1, \theta_2, \theta_3$ satisfy 
the condition \eqref{eq_theta123} 
then 
the following identity holds 
for all $f,g,h \in \calS$:  
\begin{align*}
&
\int T_{e^{i(|\xi|+|\eta|+|\xi+\eta|)} 
\theta_1 (2^{-j}\xi) \theta_2 (2^{-j}\eta)} (f,g)(x) h(x)\, dx
\\
&=
\frac{1}{(2\pi )^{2n}} 
\iint 
e^{i(|\xi|+|\eta|+|\xi+\eta|)} \, 
\theta_{1} (2^{-j} \xi) 
\theta_{2} (2^{-j} \eta) 
\theta_3 (2^{-j}(-\xi -\eta )) 
\widehat{f}(\xi) 
\widehat{g}(\eta) 
\widehat{h}(-\xi -\eta)\, d\xi d\eta 
\\
&
=
\int 
e^{i|D|} \theta_1 (2^{-j}D) f (x) 
\cdot 
e^{i|D|} \theta_2 (2^{-j}D) g (x) 
\cdot 
e^{i|D|} \theta_3 (2^{-j}D) h (x)\, 
dx.  
\end{align*}
From this identity and from the duality relation 
\eqref{eq218}, 
we see the following: 
the inequality 
\eqref{eq3000} holds 
for all $\sigma$ defined by \eqref{eq_sigma_case1} 
with $(c_j)_{j\in \N}$ satisfying $|c_j|\le 2^{jm}$ 
if and only if 
the inequality 
\begin{equation}\label{eq5000}
\begin{aligned}
& 
\sum_{j=1}^{\infty} 2^{jm} 
\bigg| 
\int 
e^{i|D|} \theta_1 (2^{-j}D) f (x) 
\cdot 
e^{i|D|} \theta_2 (2^{-j}D) g (x) 
\cdot 
e^{i|D|} \theta_3 (2^{-j}D) h (x)
\,dx
\bigg|
\\
&
\le c^{\prime}
\|\theta_1 \|_{C^M}
\|\theta_2 \|_{C^M}
\|f\|_{X}
\|g\|_{Y}
\|h\|_{Z^{\prime}} 
\end{aligned}
\end{equation} 
holds; 
and 
\eqref{eq3000} holds 
for 
$\sigma (\xi, \eta)= \theta_1 (\xi) \theta_2 (\eta)$ 
if and only if 
the inequality 
\begin{equation}\label{eq6000}
\begin{aligned}
&
\bigg| 
\int 
e^{i|D|} \theta_1 (D) f (x) 
\cdot 
e^{i|D|} \theta_2 (D) g (x) 
\cdot 
e^{i|D|} \theta_3 (D) h (x)
\,dx
\bigg|
\\
&
\le c^{\prime}
\|\theta_1 \|_{C^M}
\|\theta_2 \|_{C^M}
\|f\|_{X}
\|g\|_{Y}
\|h\|_{Z^{\prime}} 
\end{aligned}
\end{equation} 
holds  
(the constant $c$ in \eqref{eq3000} 
and the constant $c^{\prime}$ in \eqref{eq5000} 
or \eqref{eq6000} 
may not be the same).

Notice that 
if $\theta_1, \theta_2$ satisfy the 
condition of Case 1 of Subsection \ref{CMargument} 
then we can take 
$\theta_3 \in C_{0}^{\infty}$ 
so that the condition \eqref{eq_theta123} is satisfied 
and at least two of the sets 
$\supp \theta_1$, 
$\supp \theta_2$, and 
$\supp \theta_3$ 
are included in an annulus.

\subsection{Necessary conditions}
\label{necessary_condition}

Here we shall see that the estimates \eqref{eq5000} 
is also necessary for the claim \eqref{eq201} to hold.

Assume that the claim \eqref{eq201} holds and assume 
$Z^{\prime}$ is a function space 
that satisfies the duality 
relation \eqref{eq218}. 
Also assume $\theta_1, \theta_2, \theta_3$ are 
functions  in $C_{0}^{\infty}(\R^n)$ that 
satisfy \eqref{eq_theta123} and 
assume that either 
$\theta_{1}$ or $\theta_2$ has its support included in an 
annulus.

Let $(c_j)_{j\in \N}$ be a sequence of complex numbers 
such that $|c_j|\le 2^{jm}$ and 
consider the multiplier 
$\sigma$ defined by \eqref{eq_sigma_case1}.  
Then, since either $\supp \theta_1$ or $\supp \theta_2$ 
is included in an annulus, 
the support of the 
function $\theta_1 (2^{-j}\xi) \theta_2 (2^{-j}\eta)$ 
is included in the set 
$\{a2^j \le |\xi|+|\eta|\le b 2^j\}$ for  some $0<a<b<\infty$. 
From this we see that $\sigma$ belongs to the class 
$S^{m}_{1,0}(\R^{2n})$ and moreover 
the inequality 
\[
\big| 
\partial_{\xi}^{\alpha}
\partial_{\eta}^{\beta}
\sigma (\xi, \eta)
\big| 
\le C_{\alpha, \beta} 
\|\theta_1\|_{C^{|\alpha|}} 
\|\theta_2\|_{C^{|\beta|}} 
(1+ |\xi|+|\eta|)^{m-|\alpha|-|\beta|}
\]
holds with $C_{\alpha, \beta}$ 
depending only on $\alpha, \beta, n, m, a$, and $b$.  
Hence, from the claim \eqref{eq201} and 
from the uniform boundedness principle,   
it follows that $\sigma$ satisfies the 
inequality \eqref{eq3000} with $c$ 
depending only on $n, m, X, Y, Z, a$, and $b$, 
and with $M$ 
depending only on $n, m, X, Y$, and $Z$. 
Hence, 
by the duality argument given in Subsection \ref{Duality}, 
it follows that 
inequality  
\eqref{eq5000} holds 
with $c^{\prime}$ 
depending only on $n, m, X, Y, Z, Z^{\prime}, a$, and $b$.  
Thus \eqref{eq5000} is a necessary condition for \eqref{eq201}.

In particular, 
the inequality 
\begin{align*}
&
2^{jm}\bigg| 
\int 
e^{i|D|} \theta_1 (2^{-j}D) f (x) 
\cdot 
e^{i|D|} \theta_2 (2^{-j}D) g (x) 
\cdot 
e^{i|D|} \theta_3 (2^{-j}D) h (x)
\,dx
\bigg|
\\
&
\le c
\|\theta_1 \|_{C^M}
\|\theta_2 \|_{C^M}
\|f\|_{X}
\|g\|_{Y}
\|h\|_{Z^{\prime}} 
\end{align*}
with $c$ independent of $j\in \N$ 
is a necessary condition for \eqref{eq201}.

\section{Estimates for kernels}
\label{kernel}

\begin{lem} 
\label{lem_301} 
Suppose $\psi \in C_{0}^{\infty}(\R^n)$ and  
$\supp \psi \subset \{a \le |\xi|\le b\}$ 
with some $0<a<b<\infty$.  
Then for every $N>0$ there exist 
$c\in (0,\infty)$ and $M\in \N$ depending only on 
$n, a, b$, and $N$ such that 
\[
\big| 
\big( e^{i |\xi|} \psi (2^{-j}\xi)\big)^{\vee} (x) 
\big| 
\le 
c \|\psi\|_{C^M} 
2^{j (n+1)/2} \big( 
1+ 2^{j} |1-|x||\big)^{-N} 
\]
for all $j\in \N$ and $x\in \R^n$. 
\end{lem}

\begin{proof} 
We write 
\[
f_{j} (x) =
\big( 
e^{i |\xi|} 
\psi ( 2^{-j}\xi ) 
\big)^{\vee} (x)
=
\frac{1}{(2\pi)^n} 
\int_{\R^n} 
e^{i x \cdot \xi}\, e^{i |\xi|} \, 
\psi ( 2^{-j}\xi ) 
\, d\xi, 
\quad 
j \in \N, 
\;  
x \in \R^n. 
\]

The proof is easy in the case $n=1$. 
In fact, 
if $n=1$ and if we write 
$\psi_1 (\xi)=\psi (\xi) \ichi_{\{\xi >0\}}$ and 
$\psi_2 (\xi)=\psi (\xi) \ichi_{ \{\xi <0\} }$, 
then 
\begin{align*}
f_{j} (x) 
&= 
\big( e^{i \xi} 
\psi_1 ( 2^{-j}{\xi} ) 
\big)^{\vee}(x)
+
\big( e^{-i \xi} 
\psi_2 ( 2^{-j}{\xi} ) 
\big)^{\vee}(x) 
\\
&
=
2^j 
(\psi_1)^{\vee} \big( 2^j (x+1)\big) 
+ 2^j 
(\psi_2)^{\vee} \big( 2^j (x-1)\big).  
\end{align*}
Both $\psi_1$ and $\psi_2$ 
are functions in the class $C_{0}^{\infty}(\R)$ 
and their inverse Fourier transforms 
are functions in the Schwartz class,  
and thus 
\begin{equation*}
|f_{j} (x) |
\lesssim 
2^{j} \big(1+ 2^j |x+1|\big)^{-N}
+
2^{j} \big(1+ 2^j |x-1|\big)^{-N}
\approx 
2^{j} \big(1+ 2^j \big|1 - |x|\big|\big)^{-N}. 
\end{equation*}

In the rest of the argument,  
we assume $n\ge 2$.

To estimate $f_j (x)$, we recall the analysis of 
Seeger--Sogge--Stein \cite{SSS}. 
For $j\in \N$, 
take points $\{\xi_{j}^{\nu}\}_{\nu}$ such that 
\begin{align*}
&
\xi_{j}^{\nu} \in S^{n-1}=\{\xi \in \R^n \mid |\xi|=1\}, 
\\
& 
\nu \neq \nu^{\prime}\; 
\Rightarrow \; 
\big| \xi_{j}^{\nu} - \xi_j^{\nu^\prime} \big|
\ge 2^{-j/2}, 
\\
&
S^{n-1}=
\bigcup_{\nu} 
\big\{ \xi \in S^{n-1} \mid 
|\xi - \xi_{j}^{\nu}| < 2^{-j/2} \big\}, 
\end{align*}
and take functions $\{\chi_{j}^{\nu}\}_{\nu}$ such that 
\begin{align*}
&
\text{$\chi_{j}^{\nu}$ is 
a smooth positively homogeneous function 
of degree $0$  
on $\R^n \setminus \{0\}$}, 
\\
&
\chi_{j}^{\nu} (\xi) \neq 0 \; 
\Rightarrow \; 
 \bigg| \frac{\xi}{|\xi|} - \xi_{j}^{\nu}
 \bigg|
 < 2 \cdot 2^{-j/2}, 
\\
&
|\partial_{\xi}^{\alpha} 
\chi_{j}^{\nu} (\xi)|
\le C_{\alpha}\, (2^{j/2})^{|\alpha|} |\xi|^{-|\alpha|},  
\\
&
\sum_{\nu} \chi_{j}^{\nu} (\xi)=1 
\;\; \text{for all}\;\;
\xi \in \R^{n}\setminus \{0\}. 
\end{align*}
The index $\nu$ runs over an index set of cardinality 
$\approx (2^{j/2})^{n-1}$. 
Using this partition of unity 
$\{\chi_{j}^{\nu}\}_{\nu}$, we decompose 
$f_{j}(x)$ as 
\begin{align*}
& 
f_{j}(x)
=
\sum_{\nu} 
f_{j}^{\nu}(x), 
\\
& 
f_{j}^{\nu}(x)
=
\big( 
e^{i |\xi| } 
\psi ( 2^{-j}{\xi} ) 
\chi_{j}^{\nu} (\xi)
\big)^{\vee}(x)
=\frac{1}{(2\pi)^n} 
\int_{\R^n} 
e^{i (\xi\cdot x + |\xi|)} \, 
\psi ( 2^{-j} \xi ) 
\chi_{j}^{\nu}(\xi) 
\, d\xi. 
\end{align*}
We write the phase function of 
the oscillating factor 
$e^{i(\xi \cdot x + |\xi|)}$ as 
\begin{align*}
&\xi \cdot x + |\xi|
=
\xi \cdot \big(x + \xi_{j}^{\nu} \big)
+
h_{j}^{\nu}(\xi), 
\\
&
h_{j}^{\nu}(\xi)
=
\xi \cdot \bigg(
\frac{\xi}{|\xi|}
-
\xi_{j}^{\nu} 
\bigg), 
\end{align*}
and write $f_{j}^{\nu}(x)$ as 
\[
f_{j}^{\nu}(x)=
\frac{1}{(2\pi)^{n}}
\int_{\R^n} 
e^{i \xi\cdot (x+ \xi_{j}^{\nu})} \, 
\psi ( 2^{-j} {\xi} ) 
\chi_{j}^{\nu}(\xi) 
e^{i h_{j}^{\nu}(\xi)}\, d\xi. 
\]
The integrand on the right hand side 
has support included in the set 
\[
E_{j}^{\nu}
=
\bigg\{ \xi \in \R^n 
\;\bigg|\; 
2^{j} a\le |\xi| \le 2^{j} b, 
\;\; 
\bigg| \frac{\xi}{|\xi|} - \xi_{j}^{\nu} \bigg| \le 2\cdot 2^{-j/2} 
\, \bigg\}. 
\]
The following estimates hold 
for $\xi \in E_{j}^{\nu}$: 
\begin{align*}
&
\big|
\partial_{\xi}^{\alpha} 
\big[ 
\psi ( 2^{-j} {\xi} ) 
\chi_{j}^{\nu}(\xi) 
e^{i h_{j}^{\nu}(\xi)} 
\big]
\big|
\le c_{\alpha} 
\|\psi \|_{C^{|\alpha|}} 
(2^{j/2})^{-|\alpha|}, 
\\
&
\big|
(\xi_{j}^{\nu} \cdot \nabla_{\xi})^{k} 
\big[ 
\psi ( 2^{-j} {\xi} ) 
\chi_{j}^{\nu}(\xi) 
e^{i h_{j}^{\nu}(\xi)} 
\big] 
\big|
\le c_{k} \|\psi\|_{C^k} 
(2^{j})^{-k}. 
\end{align*}
The Lebesgue measure of $E_{j}^{\nu}$ 
satisfies 
$|E_{j}^{\nu}|\approx (2^{j})^{\frac{n+1}{2}}$. 
Hence by integration by parts we obtain 
the following two estimates: 
\begin{align*}
&
|f_{j}^{\nu}(x)|
\le 
c_{L} \|\psi\|_{C^{L}}
(2^{j})^{\frac{n+1}{2}}
\big(1+2^{j/2} |x+ \xi_{j}^{\nu}|\big)^{-L}, 
\\
&
|f_{j}^{\nu}(x)|
\le c_{L} \|\psi\|_{C^{L}} 
(2^{j})^{\frac{n+1}{2}}
\big(1+2^j |\xi_{j}^{\nu} \cdot (x+ \xi_{j}^{\nu})|\big)^{-L}, 
\end{align*}
where $L\in \N$ 
can be taken arbitrarily large. 
Putting the above two 
estimates together, 
we write 
\begin{equation*}
|f_{j}^{\nu}(x)|
\le 
c_{L} \, \|\psi\|_{C^{L}}\, 
(2^{j})^{\frac{n+1}{2}} g_{j}^{\nu}(x)
\end{equation*}
with
\begin{equation*}
g_{j}^{\nu}(x)=
\min 
\big\{
\big(1+2^{j/2} |x+ \xi_{j}^{\nu} |\big)^{-L}, 
\, 
\big(1+2^j |\xi_{j}^{\nu} \cdot (x+ \xi_{j}^{\nu} )|\big)^{-L}
\big\}. 
\end{equation*}
Remember that $g_{j}^{\nu}(x)$ depends on $L$. 
As for details of the above argument and its generalization to 
the case where $|\xi|$ is replaced by a general 
homogeneous function $\phi (\xi)$, 
see \cite{SSS} 
or \cite[Chapter IX, Section 4]{St}.

Now, to prove Lemma \ref{lem_301}, it is sufficient to show that 
if $N>0$ is given then we can take $L=L(n, N)$ 
sufficiently large so that we have 
\begin{equation}\label{eq-sum}
\sum_{\nu} 
g_{j}^{\nu}(x)
\le 
c_{N}
\big( 1+ 2^{j} \big|1- |x|\big|\big)^{-N}.  
\end{equation}

If 
$|x|\ge \frac{11}{10}$, then 
$|x+\xi_{j}^{\nu} |\approx |x|$ for all $\xi_{j}^{\nu} $ 
and hence 
\begin{equation*}
g_{j}^{\nu}(x) 
\le 
\big(1+2^{j/2} |x+ \xi_{j}^{\nu} |\big)^{-L}
\approx 
(2^{j/2} |x|)^{-L}.
\end{equation*}
Taking sum over $\nu$'s of cardinality $\approx (2^{j/2})^{n-1}$, 
we obtain 
\begin{equation*}
\sum_{\nu} 
g_{j}^{\nu}(x) 
\lesssim 
(2^{j/2})^{-L+n-1} |x|^{-L}, 
\end{equation*}
which implies \eqref{eq-sum} if $L$ is chosen sufficiently large.

If 
$|x|\le \frac{9}{10}$, then 
$|x+ \xi_{j}^{\nu} |\approx 1$ for all $\xi_{j}^{\nu} $ and hence 
\begin{equation*}
g_{j}^{\nu}(x) 
\le 
\big(1+2^{j/2} |x+ \xi_{j}^{\nu} |\big)^{-L}
\approx 
(2^{j/2})^{-L}.
\end{equation*}
Taking sum over $\nu$, we have 
\begin{equation*}
\sum_{\nu} 
g_{j}^{\nu}(x) 
\lesssim 
(2^{j/2})^{-L+n-1}, 
\end{equation*}
which implies \eqref{eq-sum} if $L$ is chosen sufficiently large.

In the rest of the argument, we assume 
$\frac{9}{10} < |x| < \frac{11}{10}$. 
To simplify notation, we write 
\begin{align*}
&|x|=1\pm \delta, 
\quad 0\le \delta <\frac{1}{10}, 
\\
&
\theta_{j}^{\nu} =
(\text{the angle between 
$-x$ and $\xi_{j}^{\nu}$}), 
\quad 0 \le \theta_{j}^{\nu} \le \pi . 
\end{align*}

If 
$\theta_{j}^{\nu} \ge 1$, then, for $x$ satisfying 
$\frac{9}{10}<
|x|<\frac{11}{10}$,  
we have 
$|x+ \xi_{j}^{\nu} |\approx 1$ and hence 
\begin{equation*}
g_{j}^{\nu}(x) 
\le 
\big(1+2^{j/2} |x+ \xi_{j}^{\nu} |\big)^{-L}
\approx 
(2^{j/2})^{-L}. 
\end{equation*}
Thus  
\begin{equation*}
\sum_{\nu: \; \theta_{j}^{\nu}\ge 1} 
g_{j}^{\nu}(x) 
\lesssim 
(2^{j/2})^{-L+n-1},  
\end{equation*}
which satisfies 
the bound of \eqref{eq-sum} if $L$ is sufficiently large. 
Thus in the rest of the argument we shall consider 
$\nu$'s that satisfy $\theta_{j}^{\nu}< 1$.

Observe that 
\begin{equation}\label{eq-aaa}
|x+\xi_{j}^{\nu}|
=
\big( 
1 + 2 x \cdot \xi_{j}^{\nu} + |x|^2
\big)^{1/2}
=
\{(1-|x|)^2 + 2 |x| (1- \cos \theta_{j}^{\nu} )\}^{1/2}
\approx 
\delta+ \theta_{j}^{\nu}. 
\end{equation}

On the other hand we have 
\begin{equation}\label{eq-ineqineq}
\mp \delta + \frac{1}{3} (\theta^{\nu}_{j})^2
\le \xi_j^\nu \cdot (x+\xi_j^\nu) 
\le \mp \delta + \frac{2}{3} (\theta^{\nu}_{j})^2. 
\end{equation}
To see this, 
observe that 
Taylor expansion gives 
\begin{equation*}
\cos \theta_j^{\nu}  
=
1 - \frac{(\theta_j^{\nu})^2}{2} 
+ R_{4}(\theta_j^{\nu})
\end{equation*}
with
\[
\big| R_{4}(\theta_j^{\nu})\big|
\le 
\frac{(\theta_j^{\nu})^4 }{4!} 
+
\frac{(\theta_j^{\nu})^6 }{6!}
+\cdots  
\le 
\frac{1}{12} (\theta_j^{\nu})^2 
\quad \text{for}\;\; 
0\le \theta^{\nu}_{j}< 1. 
\] 
Thus 
\begin{align*}
&\xi_j^{\nu}\cdot (x+\xi_j^{\nu})
=
1-|x|\cos \theta_j^{\nu} 
=
1- (1\pm \delta ) \bigg(1 - \frac{(\theta_j^{\nu})^2}{2} 
+ R_{4}(\theta_j^{\nu}) \bigg)
\\
&
=
\mp \delta 
+  \frac{(\theta_j^\nu)^2}{2} 
\pm \delta \frac{(\theta_j^\nu)^2}{2} 
- (1 \pm \delta)  R_{4}(\theta_j^{\nu})  
\end{align*}
and 
\begin{equation*}
\bigg| \pm \delta  \frac{(\theta_j^\nu)^2}{2}
- (1 \pm \delta)  R_{4}(\theta_j^{\nu})  \bigg|
\le 
\frac{1}{10}\cdot \frac{1}{2}(\theta_j^\nu)^2 
+ \frac{11}{10}\cdot \frac{1}{12}(\theta_j^\nu)^2 
\le 
\frac{1}{6}(\theta_j^\nu)^2 , 
\end{equation*}
from which \eqref{eq-ineqineq} follows.

From \eqref{eq-ineqineq}, we see that 
\begin{equation}\label{eq-theta<delta1/2}
0\le \delta <\frac{1}{10}, 
\;\; 
0\le \theta_{j}^{\nu} < 1, 
\;\; 
\theta_j^\nu \le \delta^{1/2}
\; 
\Rightarrow \; 
|\xi_j^\nu \cdot (x+\xi_j^\nu)|
\approx 
\delta . 
\end{equation}

Now, to estimate $\sum_{\nu:\, \theta_{j}^{\nu} < 1} 
g_{j}^{\nu}(x)$, we divide $\nu$'s 
into two sets: 
$\theta_j^\nu > \delta^{1/2}$ and 
$\theta_j^\nu \le \delta^{1/2}$.

Firstly consider $\nu$'s that satisfy 
$\theta_{j}^{\nu} > \delta^{1/2}$. 
For these $\nu$, 
we use \eqref{eq-aaa} to see that 
\begin{equation*}
g_{j}^{\nu}(x)
\le 
\big(1+2^{j/2} |x+\xi_{j}^{\nu} |\big)^{-L}
\approx 
(1+ 2^{j/2} (\theta_{j}^{\nu} + \delta) )^{-L}
\approx 
(1+ 2^{j/2} \theta_{j}^{\nu})^{-L}. 
\end{equation*}
To compute the sum of the last quantity, 
we classify $\nu$'s by the 
inequality 
\begin{equation*}
2^{-j/2} (m-1)\le  \theta_{j}^{\nu} < 2^{-j/2} m, 
\quad m\in \N. 
\end{equation*}
Since we are considering $\nu$'s satisfying 
$\theta_{j}^{\nu} > \delta^{{1}/{2}}$, 
we need only 
$m$ that satisfy $m > 2^{j/2}\delta^{{1}/{2}}$. 
Since the points 
$\xi_j^\nu$ are on the $(n-1)$-dimensional 
unit sphere $S^{n-1}$ and they are $2^{-j/2}$-separated, 
we have 
\begin{equation*}
\card \{\nu 
\mid 
2^{-j/2} (m-1)\le  \theta_{j}^{\nu} < 2^{-j/2} m 
\}
\lesssim 
\frac{
(2^{-j/2} m)^{n-2} 2^{-j/2}
}{
(2^{-j/2})^{n-1}
}
=m^{n-2}. 
\end{equation*}
Hence 
\begin{align*}
\sum_{\nu:\, 1>\theta_{j}^{\nu} > \delta^{{1}/{2}}}
g_{j}^{\nu}(x)
&\lesssim 
\sum_{\nu:\, 1>\theta_{j}^{\nu} > \delta^{{1}/{2}}}
\, 
(1+ 2^{j/2} \theta_{j}^{\nu})^{-L} 
\\
&
\lesssim 
\sum_{
m\in \N, 
\, 
m> 2^{j/2} \delta^{{1}/{2}}
}
\; \; 
\sum_{\substack
{
\nu:\, 
2^{-j/2} (m-1)\le \theta_{j}^{\nu} < 2^{-j/2} m 
}}
\, 
m^{-L}
\\
&
\lesssim 
\sum_{
m\in \N, \, m> 2^{j/2} \delta^{{1}/{2}} 
}
\, 
m ^{-L} m^{n-2}
\approx 
(1+2^{j/2} \delta^{{1}/{2}} )^{-L+n-1}. 
\end{align*}
This satisfies the bound of \eqref{eq-sum} 
if $L$ is chosen sufficiently large.

Next consider $\nu$'s that 
satisfy 
$\theta_{j}^{\nu} \le \delta^{{1}/{2}}$. 
For these $\nu$'s, we use 
\eqref{eq-theta<delta1/2} to see that 
\begin{equation*}
g_{j}^{\nu} (x) 
\le 
\big(1+ 2^{j} |\xi_j^\nu \cdot (x+ \xi_j^\nu)| \big)^{-L}
\approx 
(1+ 2^{j} \delta)^{-L}. 
\end{equation*}
Since the points 
$\xi_j^\nu$ are  on $S^{n-1}$ and $2^{-j/2}$-separated, 
we have 
\begin{equation*}
\card \big\{
\nu 
\mid 
\theta_{j}^{\nu} \le \delta^{{1}/{2}} 
\big\}
\lesssim 
\frac{
(\delta^{{1}/{2}} )^{n-1} 
}
{
(2^{-j/2})^{n-1}
}
=\big( 
2^{j/2} \delta^{{1}/{2}}
\big)
^{n-1}. 
\end{equation*}
Hence 
\begin{align*}
\sum_{\nu:\, \theta_{j}^{\nu} \le \delta^{1/2}}
g_{j}^{\nu}(x)
&
\lesssim 
\sum_{\nu:\, \theta_{j}^{\nu} \le \delta^{1/2}}
(1+ 2^{j} \delta)^{-L}
\\
&
\lesssim 
\big(
1+2^{j} \delta \big)^{-L}\, 
\big( 
2^{j/2} \delta^{{1}/{2}}
\big)^{n-1}
<
\big(1+2^{j} \delta \big)^{-L+(n-1)/2}, 
\end{align*}
which 
satisfies the bound of \eqref{eq-sum} 
if $L$ is chosen sufficiently large. 
This completes the proof of 
Lemma \ref{lem_301}.
\end{proof}

\begin{lem}\label{lem_302}
Suppose $\theta \in C_{0}^{\infty}(\R^n)$ and 
$\supp \theta \subset \{|\xi|\le a\}$ with $a\in (0,\infty)$.  
Let $\zeta$ be the function of Notation \ref{notation}. 
Then for every $N>0$ there exist 
$c\in (0,\infty)$ and $M\in \N$ depending only on 
$n, a$, and $N$ such that 
\begin{equation*}
\big| 
\big( e^{i |\xi|} \zeta  (\xi) \theta (2^{-j} \xi) 
\big)^{\vee} (x) 
\big| 
\le 
c \|\theta \|_{C^M}
\begin{cases}
{|x|^{-N}} & \text{if $|x|>2$}, 
\\
{2^{j (n+1)/2} \big( 
1+ 2^j |1-|x||
\big)^{-(n+1)/2}
} 
& 
\text{if $|x|\le 2$}, 
\end{cases}
\end{equation*}
for all $j\in \N$ and $x\in \R^n$. 
\end{lem}

\begin{proof} 
Take a nonnegative integer $j_0$ such that 
$a \le 2^{j_0}$. 
Then 
we have 
\[
\zeta (\xi) \theta (2^{-j}\xi) 
=
\sum_{k=1}^{j+ j_0} \psi (2^{-k} \xi) \theta (2^{-j}\xi) 
\]
and hence 
\[
\big( e^{i |\xi|} \zeta  (\xi) \theta (2^{-j} \xi) 
\big)^{\vee} (x) 
=
\sum_{k=1}^{j+j_0} 
\big( e^{i |\xi|} \psi (2^{-k} \xi) 
\theta (2^{-j} \xi) 
\big)^{\vee} (x).  
\]
For each term on the right hand side, 
we use Lemma \ref{lem_301} to see that 
\begin{align*}
\big| 
\big( e^{i |\xi|} \psi (2^{-k} \xi) 
\theta (2^{-j}\xi) 
\big)^{\vee} (x) \big| 
&
\lesssim  
\|\psi (\cdot ) \theta (2^{k-j} \cdot )\|_{C^M} 
2^{k(n+1)/2} 
\big( 
1+ 2^{k} |1-|x||\big)^{-L} 
\\
&
\lesssim  
\|\theta \|_{C^M}
2^{k(n+1)/2} 
\big( 
1+ 2^{k} |1-|x||\big)^{-L},  
\end{align*}
where the second $\lesssim$ 
holds because $k\le j+j_0$,  
and $L$ can be taken 
arbitrarily large. 
Taking $L$ sufficiently large and taking sum over 
$k=1, \dots , j+j_0$, 
we obtain the desired estimate. 
\end{proof}

\begin{lem} \label{lem_303}
Let $\theta \in C_{0}^{\infty}(\R^n)$ and 
let $\varphi$ be the function of Notation \ref{notation}. 
Then 
there exist 
$c\in (0,\infty)$ depending only on 
$n$ such that 
\begin{equation*}
\big| 
\big( e^{i |\xi|} \varphi (\xi) \theta (2^{-j} \xi) 
\big)^{\vee} (x) 
\big| 
\le 
c 
{\|\theta \|_{C^{n+2}} (1+|x|)^{-n-1}}
\end{equation*}
for all $j\in \N$ and $x\in \R^n$. 
\end{lem}

\begin{proof} 
Consider in general a function 
$\widetilde{\theta}\in C_{0}^{\infty}(\R^n)$ 
with $\supp \widetilde{\theta} \subset \{|\xi|\le 2\}$. 
Then 
we have 
\[
\big| 
\partial_{\xi}^{\alpha} 
\big[ (e^{i |\xi|} -1 ) \widetilde{\theta} (\xi) \big]
\big| 
\lesssim 
\|\widetilde{\theta} \|_{C^{|\alpha|}} |\xi|^{1- |\alpha|}. 
\]
Hence, if $\psi$ is the function of Notation \ref{notation}, 
then 
\[
\big| 
\partial_{\xi}^{\alpha} 
\big[ (e^{i |\xi|} -1 ) \widetilde{\theta} (\xi) \psi (2^{-k} \xi) 
\big]
\big| 
\lesssim 
\|\widetilde{\theta} \|_{C^{|\alpha|}} (2^{k})^{1- |\alpha|}
\]
for all $k\in \Z$ satisfying $k \le 1$. 
Taking inverse Fourier transform, we see that 
\[
\big| 
\big[ (e^{i |\xi|} -1 ) \widetilde{\theta} (\xi) \psi (2^{-k} \xi) 
\big]^{\vee}(x)
\big| 
\lesssim 
\|\widetilde{\theta} \|_{C^{N}} 2^{k (1+n)} 
(1+2^k |x|)^{-N}. 
\]
Taking $N=n+2$ and taking 
sum over $k\le 1$, we obtain 
\begin{align*}
\big| 
\big( (e^{i |\xi|} -1 ) 
\widetilde{\theta} (\xi) 
\big)^{\vee} (x) \big| 
&\lesssim 
\sum_{k=-\infty}^{1} 
\|\widetilde{\theta} \|_{C^{n+2}}\,  2^{k (1+n)} 
(1+2^k |x|)^{-n-2}
\\
&
\approx  
\|\widetilde{\theta}\|_{C^{n+2}} 
(1+|x|)^{-n-1}. 
\end{align*}
On the other hand, 
we also have 
\[
\big( \widetilde{\theta} (\xi) 
\big)^{\vee} (x) 
\lesssim 
\|\widetilde{\theta}\|_{C^{n+1}} 
(1+|x|)^{-n-1}. 
\]
Combining the two inequalities, we have 
\[
\big| 
\big( e^{i |\xi|} \, \widetilde{\theta} (\xi) 
\big)^{\vee} (x) \big| 
\lesssim 
\|\widetilde{\theta}\|_{C^{n+2}} 
(1+|x|)^{-n-1}. 
\]

Applying the last inequality to 
$\widetilde{\theta} (\xi) = \varphi (\xi) 
\theta (2^{-j} \xi)$, $j\in \N $, we obtain the 
desired inequality 
since 
$
\big\| 
\varphi (\cdot) 
\theta (2^{-j} \cdot )\big\|_{C^M}
\lesssim 
\big\| 
\theta \big\|_{C^M}
$ for $j\in \N$. 
\end{proof}

\begin{lem} 
\label{lem_304} 
Let 
$\theta \in C_{0}^{\infty}(\R^n)$ and suppose  
$\supp \theta \subset \{|\xi|\le a\}$ with $a\in (0, \infty)$. 
If $n=1$ and $1<p\le \infty$, 
or if 
$n\ge 2$ and $1\le p\le \infty$, 
then 
there exist 
$c\in (0,\infty)$ and $M\in \N$ 
depending only on 
$n$, $a$, and $p$ such that 
\begin{equation*}
\big\| 
\big( e^{i |\xi|} \theta (2^{-j} \xi) 
\big)^{\vee} 
\big\|_{L^p} 
\le 
c 
\|\theta \|_{C^{M}} 2^{j ((n+1)/2 - 1/p)}
\end{equation*}
for all $j\in \N$. 
\end{lem}

\begin{proof} 
This can be seen from 
the inequalities of Lemmas 
\ref{lem_302} and \ref{lem_303} by simple integration. 
\end{proof}

\begin{rem}\label{rem_impossiblecase}
(1) 
The inequality 
of Lemma \ref{lem_304} does not hold in the case $n=p=1$.  
This can be seen from the fact that in the $1$-dimensional 
case 
$\big( e^{i |\xi|}\big)^{\vee}$ is not a finite 
complex measure on $\R$. 

(2) If we replace the function $|\xi|$ by 
a function $\phi (\xi)$ which is real-valued, 
positively homogeneous of degree one, 
and $C^{\infty}$ away from the origin, 
then the claim of Lemma \ref{lem_304} still 
holds if $n=1$ and $1< p \le \infty$ or if 
$n\ge 2$ and $1\le p\le 2$, 
but 
it does not hold in general for $n\ge 2$ and $2< p\le \infty$. 
In fact, in the extreme case $\phi (\xi)\equiv 0$, 
we have $\big( e^{i \phi (\xi)} 
\theta (2^{-j} \xi) \big)^{\vee} (x)
= 2^{jn} (\theta)^{\vee} (2^{j} x)$, 
whose $L^{p}$-norm is $\approx 2^{j(n-n/p)}$,  
except for the trivial $\theta$, 
and $n-n/p > (n+1)/2 -1/p$ if $n\ge 2$ and $2<p\le \infty$. 
\end{rem}

\section{Proof of Theorem \ref{thm_Main}}
\label{proofThmMain}

In this section, we shall prove Theorem \ref{thm_Main}. 
Throughout this section, we assume 
$n \ge 2$ and $m=-(n+1)/2$.

\subsection{Scheme of the proof}
\label{scheme_of_proof}

In order to prove Theorem \ref{thm_Main}, 
we shall prove the following two inequalities 
for all 
$\theta_1, \theta_2, \theta_3 
\in C_{0}^{\infty}(\R^n)$: 
\begin{equation}\label{eq_form1}
\begin{aligned}
& 
\bigg| 
\int 
e^{i|D|} \theta_1 (D) f (x) 
\cdot 
e^{i|D|} \theta_2 (D) g (x) 
\cdot 
e^{i|D|} \theta_3 (D) h (x)
\,dx
\bigg|
\\
&
\le c 
\|\theta_1 \|_{C^M}
\|\theta_2 \|_{C^M}
\|\theta_3 \|_{C^M}
\|f\|_{L^{\infty}}
\|g\|_{L^{\infty}}
\|h\|_{h^1}, 
\end{aligned}
\end{equation} 
\begin{equation}\label{eq_form2}
\begin{aligned}
& 
\sum_{j=1}^{\infty} 2^{jm} 
\bigg| 
\int 
e^{i|D|} \theta_1 (2^{-j}D) f (x) 
\cdot 
e^{i|D|} \theta_2 (2^{-j}D) g (x) 
\cdot 
e^{i|D|} \theta_3 (2^{-j}D) h (x)
\,dx
\bigg|
\\
&
\le c 
\|\theta_1 \|_{C^M}
\|\theta_2 \|_{C^M}
\|\theta_3 \|_{C^M}
\|f\|_{L^{\infty}}
\|g\|_{L^{\infty}}
\|h\|_{h^1}.  
\end{aligned}
\end{equation}
In both inequalities, 
$M$ is a positive integer 
depending only on $n$, 
and $c$ is a positive constant 
depending only on 
$n$, 
$\supp \theta_1$, 
$\supp \theta_2$, and 
$\supp \theta_3$. 
Observe that, by virtue of the arguments of 
Subsections \ref{CMargument} and \ref{Duality}, 
the conclusion of Theorem \ref{thm_Main} 
for all three cases, 
$L^{\infty} \times L^{\infty} \to bmo$, 
$h^{1} \times L^{\infty} \to L^{1}$, and 
$L^{\infty} \times h^{1} \to L^{1}$, 
follow from these two inequalities. 

As for the inequality 
\eqref{eq_form2}, 
according to the argument of Subsections 
\ref{CMargument} and \ref{Duality}, 
it is sufficient to show it 
under the assumption that 
at least two of the functions 
$\theta_1, \theta_2, \theta_3$ 
have their supports 
included in an annulus. 
But in our proof of Theorem \ref{thm_Main} 
we don't use this additional assumption.

We shall simplify notation. 
We take a sufficiently large $M\in \N$.  
By homogeneity, we may assume 
$\|\theta_1\|_{C^M}=
\|\theta_2\|_{C^M}=\|\theta_3\|_{C^M}=1$.  
For $j\in \N \cup \{0\}$, we write 
\begin{equation*}
S^{\w}_{j} f =e^{i|D|} \theta_1 (2^{-j}D) f, 
\quad 
S^{\w}_{j} g =e^{i|D|} \theta_2 (2^{-j}D) g, 
\quad 
S^{\w}_{j} h =e^{i|D|} \theta_3 (2^{-j}D) h. 
\end{equation*}
Notice that 
the above operators $S^{\w}_{j} $ are 
not the same one; 
we did not distinguish  
$\theta_1$, $\theta_2$, and $\theta_3$. 
We shall use this rough notation since 
it will cause no significant confusion in our argument. 
Thus our task is to prove 
the inequalities  
\[
\bigg| \int 
S^{\w}_{0} f (x) \cdot S^{\w}_{0} g  (x) \cdot S^{\w}_{0} h(x)\, dx 
\bigg| 
\le c 
 \|f\|_{L^{\infty}}
 \|g\|_{L^{\infty}}
 \|h\|_{h^{1}}
\]
and 
\[
\sum_{j=1}^{\infty} 
2^{jm}
\bigg| \int 
S^{\w}_{j} f (x)\cdot  S^{\w}_{j} g  (x) \cdot S^{\w}_{j} h(x)\, dx 
\bigg| 
\le c 
 \|f\|_{L^{\infty}}
 \|g\|_{L^{\infty}}
 \|h\|_{h^{1}}. 
\]

Now, instead of proving 
the above inequalities,  
we shall prove the stronger 
inequalities  
\begin{equation*}
\big\| 
S^{\w}_{0} f \cdot S^{\w}_{0} g \cdot S^{\w}_{0} h
\big\|_{L^1}  
\le c 
 \|f\|_{L^{\infty}}
 \|g\|_{L^{\infty}}
 \|h\|_{h^{1}} 
\end{equation*}
and 
\begin{equation*}
\sum_{j=1}^{\infty} 
2^{jm}
\big\| 
S^{\w}_{j} f\cdot  S^{\w}_{j} g \cdot S^{\w}_{j} h 
\big\|_{L^1} 
\le c 
 \|f\|_{L^{\infty}}
 \|g\|_{L^{\infty}}
 \|h\|_{h^{1}}, 
\end{equation*} 
which certainly imply the desired inequalities for the integrals 
$\int S^{\w}_{j} f\cdot  S^{\w}_{j} g \cdot S^{\w}_{j} h \, dx$.

We make another reduction. 
By virtue of the atomic decomposition of $h^1$ 
and translation invariance of the situation, 
it is sufficient to prove the above $L^1$-norm inequalities 
in the case where $h$ is an $h^1$-atom supported on 
a ball centered at the origin 
(as for the atomic decomposition of $h^1$, see Goldberg \cite{Goldberg}). 
Hence we assume $h$ satisfies the following condition: 
\begin{equation*}
\left\{ 
\; 
\begin{aligned}
& h \in L^{\infty} (\R^n), 
\\
& \supp h \subset \{|x|\le r\} 
\;\; \text{with some} \;\; r \in (0,1], 
\\
& \|h\|_{L^{\infty}}\le r^{-n}, 
\\
& \int h(x)\, dx = 0 
\;\; \text{if}\;\; r<1.
\end{aligned}
\right.
\end{equation*}

Under the above assumptions, 
our goal is to prove the 
inequalities  
 \begin{equation}\label{eq414}
\big\| 
S^{\w}_{0} f \cdot S^{\w}_{0} g \cdot S^{\w}_{0} h
\big\|_{L^1}  
\le c 
 \|f\|_{L^{\infty}}
 \|g\|_{L^{\infty}}
\end{equation}
and 
\begin{equation}\label{eq415}
\sum_{j=1}^{\infty} 
2^{jm}
\big\| 
S^{\w}_{j} f\cdot  S^{\w}_{j} g \cdot S^{\w}_{j} h 
\big\|_{L^1} 
\le c 
 \|f\|_{L^{\infty}}
 \|g\|_{L^{\infty}}. 
\end{equation}

To prove \eqref{eq414} and \eqref{eq415}, we will use 
the kernels of the operators $S^{\w}_{j}$. 
We write 
\[
K^{\w}_{j}(x) = \big( e^{i|\xi|} \theta_{\ell} (2^{-j} \xi) \big)^{\vee} (x), 
\quad 
\ell = 1, 2, 3. 
\]
Thus $S^{\w}_{j} F = K^{\w}_{j} \ast F$. 
Notice that here again our notation is rough; 
we did not distinguish the three functions $\theta_{\ell}$, 
$\ell = 1, 2, 3$. 
But the above notation will cause no confusion. 
The estimates of the  kernel $K^{\w}_{j}$ are given in 
Section \ref{kernel}.

\subsection{Proof of \eqref{eq414}}
\label{proof_eq414}

By Lemma \ref{lem_304}, we have $\|K^{\w}_{0}\|_{L^1}\lesssim 1$, 
and thus 
\begin{align*}
&
\|S^{\w}_{0} f \|_{L^{\infty}} \le \|K^{\w}_{0}\|_{L^1} \|f\|_{L^{\infty}} 
\lesssim \|f\|_{L^{\infty}} , 
\\
&
\|S^{\w}_{0} g \|_{L^{\infty}} \le \|K^{\w}_{0}\|_{L^1} \|g\|_{L^{\infty}} 
\lesssim \|g\|_{L^{\infty}} , 
\\
&
\|S^{\w}_{0} h \|_{L^{1}} \le \|K^{\w}_{0}\|_{L^1} \|h\|_{L^{1}} 
\lesssim \|h\|_{L^{1}} \lesssim 1. 
\end{align*}
Hence H\"older's inequality gives 
\[
\| S^{\w}_{0} f\cdot S^{\w}_{0} g\cdot S^{\w}_{0} h \|_{L^1}
\le 
\|S^{\w}_{0} f \|_{L^{\infty}}
\|S^{\w}_{0} g \|_{L^{\infty}} 
\|S^{\w}_{0} h \|_{L^{1}}  
\lesssim 
\|f\|_{L^{\infty}} \|g\|_{L^{\infty}}.  
\]
Thus \eqref{eq414} is proved.

\subsection{Reduction of the proof of \eqref{eq415}}
\label{reduction_of_proof_eq415}

To prove \eqref{eq415}, 
we decompose the $L^1 (\R^n)$-norm as 
\[
\| \dots \|_{L^1} = 
\| \dots \|_{L^1 (B(0,4))} 
+ 
\sum_{k=2}^{\infty}
\| \dots \|_{L^1 (E_k)}, 
\]
with 
\[
E_k
=B(0, 2^{k+1}) \setminus B(0, 2^k)
=\{x \in \R^n \mid 2^k \le |x| < 2^{k+1}\}. 
\]
We shall prove 
\begin{equation}\label{eq416}
\sum_{j=1}^{\infty} 
\sum_{k=2}^{\infty} 
2^{jm} 
\big\|  S^{\w}_{j} f\cdot S^{\w}_{j} g\cdot S^{\w}_{j} h \big\|_{L^1 (E_k)}
\lesssim 
\|f\|_{L^{\infty}}
\|g\|_{L^{\infty}}
\end{equation}
and 
\begin{equation}\label{eq417}
\sum_{j=1}^{\infty} 
2^{jm} 
\big\|  S^{\w}_{j} f\cdot S^{\w}_{j} g\cdot S^{\w}_{j} h \big\|_{L^1 (B(0,4))}
\lesssim 
\|f\|_{L^{\infty}}
\|g\|_{L^{\infty}}. 
\end{equation}

\subsection{Proof of \eqref{eq416}}
\label{proof_eq416}

For $k\ge 2$, 
we decompose $f$ and $g$ as 
\begin{align*}
&
f = f \ichi_{B(0, 10\cdot 2^{k})} + f \ichi_{B(0, 10\cdot 2^{k})^{c}} 
=f^{0}_{k} + f^{1}_{k}, 
\\
&
g = g \ichi_{B(0, 10\cdot 2^{k})} + g \ichi_{B(0, 10\cdot 2^{k})^{c}} 
=g^{0}_{k} + g^{1}_{k}, 
\end{align*}
and we shall prove 
\begin{align}
&
\big\|  S^{\w}_{j} f^{1}_{k} 
\cdot S^{\w}_{j} g^{1}_{k} 
\cdot S^{\w}_{j} h \big\|_{L^1 (E_k)} 
\lesssim 2^{-3k} \|f\|_{L^{\infty}} \|g \|_{L^{\infty}}, 
\label{eq431}
\\
&
\big\|  S^{\w}_{j} f^{1}_{k} 
\cdot S^{\w}_{j} g^{0}_{k} 
\cdot S^{\w}_{j} h \big\|_{L^1 (E_k)} 
\lesssim 2^{-2k} \|f\|_{L^{\infty}} \|g \|_{L^{\infty}}, 
\label{eq432}
\\
&
\big\|  S^{\w}_{j} f^{0}_{k} 
\cdot S^{\w}_{j} g^{1}_{k} 
\cdot S^{\w}_{j} h \big\|_{L^1 (E_k)} 
\lesssim 2^{-2k} \|f\|_{L^{\infty}} \|g \|_{L^{\infty}}, 
\label{eq433}
\\
&
\big\|  S^{\w}_{j} f^{0}_{k} 
\cdot S^{\w}_{j} g^{0}_{k} 
\cdot S^{\w}_{j} h \big\|_{L^1 (E_k)} 
\lesssim 2^{-k} \|f\|_{L^{\infty}} \|g \|_{L^{\infty}}. 
\label{eq434}
\end{align}
If these are proved, then 
taking sum over $k$ and $j$ we obtain \eqref{eq416}.

{\it Proof of\/} \eqref{eq431}.  
By Lemmas \ref{lem_302} and \ref{lem_303}, we have 
\begin{equation}\label{eq435}
|K^{\w}_{j} (x)| \lesssim |x|^{-n-1} 
\quad \text{for} \quad |x|>2. 
\end{equation}
From this inequality, 
we have 
\begin{align*}
&
\big\| 
S^{\w}_{j} h 
\big\|_{L^{1}(E_k)}
=
\int_{2^k \le |x|<2^{k+1}}
\, 
\bigg| \int_{|y|\le r}\, 
K^{\w}_{j} (x-y)
h(y)\, dy\bigg|\, dx
\\
&
\lesssim 
\iint_{\substack{
2^k \le |x|<2^{k+1}
\\
|y|\le r} }
\, 
|x-y|^{-n-1} r^{-n}\, dydx
\approx 
2^{-k}. 
\end{align*}
On the other hand, 
using the same inequality 
\eqref{eq435}, 
we have 
\begin{equation}\label{eq436}
\begin{aligned}
& 
\big\| 
S^{\w}_{j} f^{1}_{k}
\big\|_{L^{\infty}(E_k)}
=\sup_{2^k \le |x|<2^{k+1}}
\, 
\bigg| 
\int_{|y|\ge 10\cdot 2^{k}} 
\, 
K^{\w}_{j} (x-y)
 f(y)\, dy
 \bigg|
 \\
&
\lesssim 
\sup_{2^k \le |x|<2^{k+1}}
\, 
\int_{|y|\ge 10\cdot 2^{k}} 
\, 
|x-y|^{-n-1}
 \|f\|_{L^{\infty}}\, dy 
 \approx 2^{-k}\|f\|_{L^{\infty}} 
\end{aligned}
\end{equation}
and, similarly,  
\begin{equation*}
\big\| 
S^{\w}_{j} g^{1}_{k}
\big\|_{L^{\infty}(E_k)}
\lesssim 
 2^{-k}\|g\|_{L^{\infty}}. 
\end{equation*}
From the above inequalities, we obtain 
\begin{align*}
\big\|
S^{\w}_{j} f^{1}_{k}\cdot 
S^{\w}_{j} g^{1}_{k}\cdot 
S^{\w}_{j} h
\big\|_{L^1 (E_k)}
&
\le 
\big\|
S^{\w}_{j} f^{1}_{k}
\big\|_{L^{\infty} (E_k)}
\big\|
S^{\w}_{j} g^{1}_{k}
\big\|_{L^{\infty} (E_k)}
\big\|
S^{\w}_{j} h
\big\|_{L^1 (E_k)}
\\
&
\lesssim 
2^{-3k} \|f\|_{L^{\infty}}
\|g\|_{L^{\infty}}.
\end{align*}
Thus \eqref{eq431} is proved. 

{\it Proof of\/} \eqref{eq432}. 
From \eqref{eq435}, we have 
\begin{equation}\label{eq439}
\begin{aligned}
&
\big\| S^{\w}_{j} h \big\|_{L^2 (E_k)} 
=
\bigg\|
\int_{|y|\le r}\, 
K^{\w}_{j}(x-y) h(y)\, dy
\bigg\|_{L^2 (2^k \le |x|<2^{k+1})}
\\
&
\lesssim 
\bigg\|
\int_{|y|\le r}\, 
|x-y|^{-n-1}\, r^{-n} \, dy
\bigg\|_{L^2 (2^k \le |x|<2^{k+1})}
\approx 
(2^k)^{-n-1} 2^{kn/2}. 
\end{aligned}
\end{equation}
On the other hand, 
since the multiplier $e^{i|\xi|} \theta_{\ell} (2^{-j}\xi)$ is bounded, 
Plancherel's theorem gives the $L^2$-estimate 
$\|S^{\w}_{j}\|_{L^2 \to L^2}\lesssim 1$. 
Hence 
\begin{equation}\label{eq438}
\big\| S^{\w}_{j} g^{0}_{k} \big\|_{L^2} 
\lesssim 
\big\| g^{0}_{k} \big\|_{L^2} 
\lesssim 2^{kn/2} \|g\|_{L^{\infty}}. 
\end{equation}
By \eqref{eq439}, \eqref{eq438}, and \eqref{eq436}, 
and by the Cauchy--Schwarz inequality, 
we obtain 
\begin{align*}
&\big\|
S^{\w}_{j} f^{1}_{k}\cdot 
S^{\w}_{j} g^{0}_{k}\cdot 
S^{\w}_{j} h
\big\|_{L^1 (E_k)}
\le 
\big\|
S^{\w}_{j} f^{1}_{k}
\big\|_{L^{\infty} (E_k)}
\big\|
S^{\w}_{j} g^{0}_{k}
\big\|_{L^{2} }
\big\|
S^{\w}_{j} h
\big\|_{L^{2} (E_k)}
\\
&
\lesssim 
2^{-k} \|f\|_{L^{\infty}} 
2^{kn/2} \|g\|_{L^{\infty}} 
(2^{k})^{-n-1} 2^{kn/2} 
=
2^{-2k}
\|f\|_{L^{\infty}}
\|g\|_{L^{\infty}}.
\end{align*}
Thus \eqref{eq432} is proved.

{\it Proof of\/} \eqref{eq433}. 
This is the same as \eqref{eq432} 
by symmetry.

{\it Proof of\/} \eqref{eq434}.  
Using the estimate \eqref{eq435}, we see that 
\begin{align*}
&
\big\| S^{\w}_{j} h \big\|_{L^{\infty} (E_k)} 
=
\sup_{2^k \le |x|<2^{k+1}}\, \
\bigg| 
\int_{|y|\le r}\, 
K^{\w}_{j}(x-y) h(y)\, dy
\bigg|
\\
&
\lesssim 
\sup_{2^k \le |x|<2^{k+1}}\, 
\int_{|y|\le r}\, 
|x-y|^{-n-1}\, r^{-n} \, dy
\approx  
(2^k)^{-n-1}. 
\end{align*}
By the same reason as \eqref{eq438}, we have 
\[
\big\| S^{\w}_{j} f^{0}_{k} \big\|_{L^2} 
\lesssim 
2^{kn/2} \|f\|_{L^{\infty}}, 
\quad 
\big\| S^{\w}_{j} g^{0}_{k} \big\|_{L^2} 
\lesssim 
2^{kn/2} \|f\|_{L^{\infty}}.  
\]
From these estimates, 
using the Cauchy--Schwarz inequality, 
we obtain 
\begin{align*}
&\big\|
S^{\w}_{j} f^{0}_{k}\cdot 
S^{\w}_{j} g^{0}_{k}\cdot 
S^{\w}_{j} h
\big\|_{L^1 (E_k)}
\le 
\big\|
S^{\w}_{j} f^{0}_{k}
\big\|_{L^{2} }
\big\|
S^{\w}_{j} g^{0}_{k}
\big\|_{L^{2} }
\big\|
S^{\w}_{j} h
\big\|_{L^{\infty} (E_k)}
\\
&
\lesssim 
2^{kn/2} \|f\|_{L^{\infty}} 
2^{kn/2} \|g\|_{L^{\infty}} 
(2^{k})^{-n-1}
=
2^{-k}
\|f\|_{L^{\infty}}
\|g\|_{L^{\infty}}.
\end{align*}
Thus \eqref{eq434} is proved and 
the proof of \eqref{eq416} is complete.

\subsection{Proof of \eqref{eq417}}
\label{proof_eq417}

To prove \eqref{eq417}, we decompose $f$ and $g$ as 
\begin{align*}
&
f = f \ichi_{B(0, 10)} + f \ichi_{B(0, 10)^{c}} 
=f^{0} + f^{1}, 
\\
&
g = g \ichi_{B(0, 10)} + g \ichi_{B(0, 10)^{c}} 
=g^{0} + g^{1}. 
\end{align*}
We shall prove 
\begin{align}
&
\big\|  S^{\w}_{j} f^{1}
\cdot S^{\w}_{j} g^{1}
\cdot S^{\w}_{j} h \big\|_{L^1 (B(0,4))} 
\lesssim 2^{j(n-1)/2} \|f\|_{L^{\infty}} \|g \|_{L^{\infty}}, 
\label{eq441}
\\
&
\big\|  S^{\w}_{j} f^{1} 
\cdot S^{\w}_{j} g^{0} 
\cdot S^{\w}_{j} h \big\|_{L^1 (B(0,4))} 
\lesssim 2^{jn/2} \|f\|_{L^{\infty}} \|g \|_{L^{\infty}}, 
\label{eq442}
\\
&
\big\|  S^{\w}_{j} f^{0} 
\cdot S^{\w}_{j} g^{1}
\cdot S^{\w}_{j} h \big\|_{L^1 (B(0,4))} 
\lesssim 2^{jn/2} \|f\|_{L^{\infty}} \|g \|_{L^{\infty}}, 
\label{eq443}
\\
&
\big\|  S^{\w}_{j} f^{0}
\cdot S^{\w}_{j} g^{0}
\cdot S^{\w}_{j} h \big\|_{L^1 (B(0,4))} 
\lesssim 
 2^{j(n+1)/2} 
 \min \{2^j r, \, (2^j r)^{-1}\}
\|f\|_{L^{\infty}} \|g \|_{L^{\infty}}, 
\label{eq444}
\end{align}
If these are proved, 
then 
multiplying them 
by $2^{jm}=2^{-j(n+1)/2}$ and 
taking sum over $j\in \N$ we obtain 
\eqref{eq417}.

{\it Proof of\/} \eqref{eq441}. 
By 
Lemma \ref{lem_304} we have  
$\|K^{\w}_{j}\|_{L^1}\lesssim 2^{j(n-1)/2}$ and hence 
\begin{equation*}
\big\| 
S^{\w}_{j} h 
\big\|_{L^{1}}
\le 
\| K^{\w}_{j}\|_{L^1}
\|h\|_{L^1}
\lesssim 
2^{j(n-1)/2}. 
\end{equation*}
On the other hand, 
using the kernel estimate 
\eqref{eq435}, we see that 
\begin{equation}\label{eq446}
\begin{aligned}
& 
\big\| 
S^{\w}_{j} f^{1}
\big\|_{L^{\infty}(B(0,4))}
=\sup_{|x|<4}
\, 
\bigg| 
\int_{|y|\ge 10} 
\, 
K^{\w}_{j} (x-y)
 f(y)\, dy
 \bigg|
 \\
&
\lesssim 
\sup_{|x|<4}
\, 
\int_{|y|\ge 10} 
\, 
|x-y|^{-n-1}
 \|f\|_{L^{\infty}}\, dy 
 \approx \|f\|_{L^{\infty}}. 
\end{aligned}
\end{equation}
Similarly we have 
$\big\| 
S^{\w}_{j} g^{1}
\big\|_{L^{\infty}(B(0,4))}
\lesssim 
\|g\|_{L^{\infty}}$. 
From these inequalities, we obtain 
\begin{align*}
\big\|
S^{\w}_{j} f^{1}\cdot 
S^{\w}_{j} g^{1}\cdot 
S^{\w}_{j} h
\big\|_{L^1 (B(0,4))}
&
\le 
\big\|
S^{\w}_{j} f^{1}
\big\|_{L^{\infty} (B(0,4))}
\big\|
S^{\w}_{j} g^{1}
\big\|_{L^{\infty} (B(0,4))}
\big\|
S^{\w}_{j} h
\big\|_{L^1 }
\\
&
\lesssim 
\|f\|_{L^{\infty}}
\|g\|_{L^{\infty}} 
2^{j(n-1)/2}. 
\end{align*}
Thus \eqref{eq441} is proved.

{\it Proof of\/} \eqref{eq442}. 
Lemma \ref{lem_304} or Plancherel's theorem 
yields $\|K^{\w}_{j}\|_{L^2}\lesssim 2^{jn/2}$ and 
hence 
\begin{equation}\label{eq1000}
\big\| S^{\w}_{j} h \big\|_{L^2 } 
\le 
\|K^{\w}_{j}\|_{L^2} \|h\|_{L^1}
\lesssim 2^{jn/2}. 
\end{equation}
By the same reason as in \eqref{eq438}, we have the $L^2$-estimate 
\begin{equation}\label{eq447}
\big\| S^{\w}_{j} g^{0} \big\|_{L^2} 
\lesssim 
\big\| g^{0} \big\|_{L^2} 
\lesssim \|g\|_{L^{\infty}}. 
\end{equation}
Using \eqref{eq1000}, \eqref{eq447}, and \eqref{eq446}, 
and using 
the Cauchy--Schwarz inequality, we obtain 
\begin{align*}
\big\|
S^{\w}_{j} f^{1}\cdot 
S^{\w}_{j} g^{0}\cdot 
S^{\w}_{j} h
\big\|_{L^1 (B(0,4))}
&
\le 
\big\|
S^{\w}_{j} f^{1}
\big\|_{L^{\infty} (B(0,4))}
\big\|
S^{\w}_{j} g^{0}
\big\|_{L^{2} }
\big\|
S^{\w}_{j} h
\big\|_{L^{2} }
\\
&
\lesssim 
\|f\|_{L^{\infty}} 
\|g\|_{L^{\infty}} 
2^{jn/2}. 
\end{align*}
Thus \eqref{eq442} is proved.

{\it Proof of\/} \eqref{eq443}. 
This is the same as \eqref{eq442} by symmetry.

{\it Proof of\/} \eqref{eq444}. 
Recall that Lemma \ref{lem_304} give the estimate 
$\|K^{\w}_{j}\|_{L^{\infty}}\lesssim 2^{j(n+1)/2}$, 
which implies the estimate 
\[
\big\| 
S^{\w}_{j} h \big\|_{L^{\infty}}
\le 
\|K^{\w}_{j}\|_{L^{\infty}} 
\|h\|_{L^1}
\lesssim 
2^{j(n+1)/2}.    
\]
In order to prove \eqref{eq444}, 
we shall improve this estimate 
to the following one: 
\begin{equation}\label{eq449}
\big\| 
S^{\w}_{j} h \big\|_{L^{\infty}}
\lesssim 
2^{j(n+1)/2} \min \{2^j r, (2^j r)^{-1}\}.    
\end{equation}
If this is proved, then we can prove 
\eqref{eq444} as follows.  
By the same reason as \eqref{eq447}, we have 
\begin{equation*}
\|S^{\w}_{j} f^{0} \|_{L^2}\lesssim \|f\|_{L^{\infty}}, 
\quad 
\|S^{\w}_{j} g^{0} \|_{L^2}\lesssim \|g\|_{L^{\infty}};   
\end{equation*}
thus 
using 
\eqref{eq449} and  
using the Cauchy--Schwarz inequality,  
we obtain 
\begin{align*}
\big\|
S^{\w}_{j} f^{0}\cdot 
S^{\w}_{j} g^{0}\cdot 
S^{\w}_{j} h
\big\|_{L^1 (B(0,4))}
&
\le 
\big\|
S^{\w}_{j} f^{0}
\big\|_{L^{2} }
\big\|
S^{\w}_{j} g^{0}
\big\|_{L^{2} }
\big\|
S^{\w}_{j} h
\big\|_{L^{\infty} }
\\
&
\lesssim 
\|f\|_{L^{\infty}} 
\|g\|_{L^{\infty}} 
2^{j(n+1)/2} 
 \min \{2^j r, (2^j r)^{-1}\}.      
\end{align*}
Thus it is sufficient to prove \eqref{eq449}.

To prove \eqref{eq449}, first consider the case $2^{j}r \le 1$. 
Then, since $r \le 2^{-j}<1$, 
the $h^1$-atom $h$ satisfies the moment condition 
$\int h(y)\, dy=0$. 
Hence we have 
\begin{align*}
S^{\w}_{j} h (x) 
&
=
\int \big( K^{\w}_{j} (x-y) - K^{\w}_{j} (x)
\big)
h(y)\, dy
\\
&
=
\iint_{\substack{
0<t<1
\\
|y|\le r
}}
\grad K^{\w}_{j} (x-ty)\cdot (-y) h(y)\, dt dy. 
\end{align*}
By Lemma \ref{lem_304}, 
we have 
\begin{align*}
\big| 
\grad K^{\w}_{j} (z) 
\big|
&
=
\big| 
\big( 
e^{i |\xi|} \theta_3 (2^{-j} \xi) \xi
\big)^{\vee} (z)
\big|
\\
&
= \big| 2^{j}
\big( 
e^{i |\xi|} \theta_3 (2^{-j} \xi) 2^{-j}\xi
\big)^{\vee} (z)
\big|
\lesssim 
2^{j} 2^{j(n+1)/2}. 
\end{align*}
Hence 
\begin{align*}
\big| 
S^{\w}_{j} h (x) 
\big| 
&
\lesssim 
\iint_{\substack{
0<t<1
\\
|y|\le r
}}
\big| \grad K^{\w}_{j} (x-ty)\big| 
|y|\, |h(y)|\, dt dy
\\
&
\lesssim 
 2^{j} 2^{j(n+1)/2} r. 
\end{align*}
This proves \eqref{eq449} in the case $2^j r \le 1$. 

Next consider the case $2^j r >1$. 
From the assumptions on $h$, 
we have 
\begin{equation}\label{eq_Sjhx}
\big| 
S^{\w}_{j}h (x)
\big|
=\big| K^{\w}_{j} \ast h (x) \big|
\le 
\int_{|x-y|\le r} 
|K^{\w}_{j} (y)|\, r^{-n}\, 
dy 
=: (\ast) . 
\end{equation}
We shall prove $(\ast)\lesssim 2^{j (n+1)/2} (2^j r)^{-1}$. 
Notice that 
Lemmas \ref{lem_302} and \ref{lem_303} give  
the following estimate 
\begin{equation}\label{eq_Kjy}
\big| K^{\w}_{j} (y) \big| 
\lesssim 
\begin{cases}
{|y|^{-n-1}} & \text{if}\;\; {|y|>2}, \\
{2^{j(n+1)/2} \big( 1 + 2^j |1-|y|| \big)^{-(n+1)/2}} 
& \text{if}\;\; {|y|\le 2}. 
\end{cases}
\end{equation}
We shall consider several situations separately.

Firstly, suppose $1/10 <r\le 1$. 
In this case, we have 
\begin{align*}
(\ast) 
\lesssim 
\int_{\R^n} 
\big| K^{\w}_{j} (y)\big|\, dy
\lesssim 2^{j (n+1)/2} 2^{-j}
\approx 
2^{j (n+1)/2} \big( 2^{j} r\big)^{-1}, 
\end{align*}
where the second $\lesssim $ 
follows from \eqref{eq_Kjy} (or it is also given in 
Lemma \ref{lem_304}).

Secondly, suppose $0<r \le 1/10$ and suppose 
$|x|\le 1/2$ or $|x| \ge 3/2$. 
In this case, 
using $\eqref{eq_Kjy}$, we see that 
\[
|x-y|\le r \; \Rightarrow \; 
|y|\le 1/2 + 1/10
\; \; 
\text{or} \; \; 
|y| \ge 3/2 -1/10 
\; 
\Rightarrow 
\; 
|K^{\w}_{j}(y)|
\lesssim 1. 
\]
Hence 
\begin{equation*}
(\ast) 
\lesssim 
\int_{|x-y|\le r} 
r^{-n}\, dy 
\approx 
1 
< 
2^{j (n+1)/2} \big( 2^{j} r\big)^{-1}. 
\end{equation*}

Finally, 
suppose $0<r \le 1/10$ and $1/2 < |x| < 3/2$. 
In this case, 
\[
|x-y|\le r \; \Rightarrow \; 
|x|-r \le |y| \le |x|+r
\;\; 
\text{and}\;\;  
\bigg| \frac{x}{|x|}- \frac{y}{|y|}\bigg| 
\lesssim r. 
\]
Hence, using \eqref{eq_Kjy} 
and using polar coordinate, we see that 
\begin{align*}
(\ast) 
& 
\lesssim 
\int_{\substack{ 
|x|-r \le |y| \le |x|+r
\\
\big| x/|x| - y/|y| \big|
\lesssim r
}}\,   
2^{j(n+1)/2} 
\big( 1+ 2^j |1-|y||\big)^{-(n+1)/2}
r^{-n}\, dy 
\\
&
\approx 
2^{j (n+1)/2}\, r^{-n}\, r^{n-1} 
\int_{|x|-r}^{|x|+r} 
\big( 1+ 2^{j} |1-t|\big)^{-(n+1)/2}\, dt
\\
&
\le 
2^{j (n+1)/2}\, r^{-n}\, r^{n-1} 
\int_{1-r}^{1+r} 
\big( 1+ 2^{j} |1-t|\big)^{-(n+1)/2}\, dt
\\
&\approx 
2^{j (n+1)/2}\, r^{-1} 2^{-j}, 
\end{align*}
where the last $\approx$ holds because $(n+1)/2 >1$. 
Thus we have proved \eqref{eq449} in the case $2^j r >1$ as well. 
Now \eqref{eq444} is proved and the proof of 
Theorem \ref{thm_Main} is complete.

\section{A remark on the condition $m=-(n+1)/2$ in 
Theorem \ref{thm_Main}}
\label{necessary_conditions_m}

Let $\widetilde{\psi}, \theta \in C_{0}^{\infty}(\R^m)$ be 
such that 
\begin{equation*}
\supp \widetilde{\psi} \subset \{1/3 \le |\xi| \le 3\}, 
\quad 
\widetilde{\psi} (\xi) = 1 \;\; \text{for}\;\; 
1/2 \le |\xi| \le 2, 
\end{equation*}
and 
\begin{equation*}
\supp \theta \subset \{|\xi| \le 10\}, 
\quad 
\theta (\xi) = 1 \;\; \text{for}\;\; 
|\xi| \le 6. 
\end{equation*}
As we have seen in Subsection \ref{necessary_condition}, 
if the claim 
\[ 
e^{i(|\xi|+|\eta|+|\xi+\eta|)} \sigma (\xi, \eta) 
\in \calM (L^{\infty}\times L^{\infty} \to BMO)
\quad 
\text{for all}
\; \;  
\sigma \in S^{m}_{1,0}(\R^{2n})  
\]
holds, 
then the inequality 
\begin{align*}
&
2^{jm} 
\bigg| \int 
e^{i|D|} \widetilde{\psi} (2^{-j}D) f (x) 
\cdot 
e^{i|D|} \widetilde{\psi} (2^{-j}D) g (x) 
\cdot 
e^{i|D|} \theta (2^{-j}D) h(x)
\, dx 
\bigg|
\\
&\le c 
\|f\|_{L^{\infty}}
\|g\|_{L^{\infty}}
\|h\|_{H^{1}}
\end{align*}
must hold with $c$ independent of $j\in \N$. 
In the proof of Theorem \ref{thm_Main}, 
in Section \ref{proofThmMain}, 
we have 
proved that the stronger inequality 
\begin{equation}\label{eq521}
2^{jm} \big\|  e^{i|D|} \widetilde{\psi} (2^{-j}D) f 
\cdot 
e^{i|D|} \widetilde{\psi} (2^{-j}D) g 
\cdot 
e^{i|D|} \theta (2^{-j}D) h 
\big\|_{L^1}
\le c 
\|f\|_{L^{\infty}}
\|g\|_{L^{\infty}}
\|h\|_{H^{1}}
\end{equation}
holds for $m=-(n+1)/2$. 
In this section, 
we shall prove that the $L^1$-inequality 
\eqref{eq521}  
cannot be extended to the case 
$m>-(n+1)/2$.

\begin{prop}\label{prop_Main2} 
Let $n \ge 2$. 
Then the estimate \eqref{eq521} holds  
only if $m \le -(n+1)/2$. 
\end{prop}

To prove this proposition, we use the following lemma. 

\begin{lem}[{\cite[Lemma 5.2]{KMT-RMI}}]\label{lem-220919}
Let $\psi$ be function in $C_{0}^{\infty}(\R^n)$ that is 
radial, supported on $\{1/2 \le |\xi|\le 2\}$, nonnegative, and 
not identically equal to $0$.  
Set 
\[
G_j (x) = 
\left( e^{i |\xi|} \psi (2^{-j} \xi)\right)^{\vee} (x).  
\]
Then 
there exist 
$\delta, c_0 \in (0, \infty)$ and $j_0\in \N$, 
depending only on 
$n$ and $\psi$, 
such that 
\begin{equation*}
\big|e^{-i\omega_n}\, 
2^{-j{(n+1)}/{2}} \, 
G_j (x)
- c_0\, \big|
\le \frac{c_0}{10}
\quad  
\text{if}
\;\; 
1- \delta 2^{-j } < |x|  <1+ \delta 2^{-j}  
\; \; \text{and}\;\; 
j > j_0, 
\end{equation*}
where $\omega_n = {(n-1)\pi}/{4}$. 
\end{lem} 

We also use the following lemma. 

\begin{lem}[Grothendieck's inequality] \label{lem_Grothendieck}
Let $I$ be a finite index set and let $a_{i,j}$ ($i,j \in I$) be 
complex numbers. 
Let $A\in (0, \infty)$ and suppose the inequality 
\[
\bigg| \sum_{i,j \in I} a_{i,j} x_i y_j \bigg| 
\le A\,  \sup |x_i| \cdot \sup |y_j|
\]
holds for all complex sequences $(x_i)_{i\in I}$ and 
$(y_j)_{j\in I}$. 
Then for any complex Hilbert space $H$ and 
for any sequences $(u_i)_{i\in I}$ and $(v_j)_{j\in I}$ 
of vectors in $H$, 
the inequality 
\[
\bigg| \sum_{i,j \in I} a_{i,j} \langle u_i, v_j \rangle \bigg| 
\le CA \, \sup \|u_i\|_{H} \cdot \sup \|v_j\|_{H}
\]
holds, where $\langle \cdot , \cdot \rangle$ denotes 
the inner product in $H$ and $C$ is a 
universal constant (independent of $I$, $(a_{i,j})$, 
$A$, and $H$). 
\end{lem}

A proof of this lemma can be found, {\it e.g.\/}, 
in \cite[Section 18.2]{Garling}.

\begin{proof}[Proof of Proposition \ref{prop_Main2}]
Let $\psi$ and $G_j$ be the functions of Lemma \ref{lem-220919}. 
In the following argument, $\delta$ and $j_0$ are the 
numbers given in Lemma \ref{lem-220919} and the letter $j$ 
denotes positive integers satisfying $j>j_0$. 
We set 
\[
F_j (x)= \big( \psi (2^{-j}\xi)\big)^{\vee}
(x). 
\]
We take a sufficiently small number $\delta^{\prime}>0$ 
and write 
\[
v_{\mu} = \delta^{\prime} 2^{-j} \mu, \quad \mu \in \Z^n, 
\]
and 
\[ 
I=\{\mu \in \Z^n \mid 1/2 < |v_{\mu}| < 3/2\}. 
\] 
Let $\alpha = (\alpha_{\lambda})_{\lambda \in I}$ and 
$\beta = (\beta_{\mu})_{\mu \in I}$ be 
sequences of complex numbers. 
We shall test the inequality \eqref{eq521} to 
the following functions: 
\begin{align*}
&
f(x) =\sum_{\lambda \in I} \alpha_{\lambda} F_j (x- v _{\lambda}), 
\\
&
g(x) =\sum_{\mu \in I} \beta_{\mu} F_j (x- v _{\mu}), 
\\
&
h(x) = F_j (x). 
\end{align*}

Since $F_j (x) = 2^{jn} (\psi^{\vee})(2^j x)$ is a bump 
function and since the points $(v_{\mu})_{\mu \in \Z^n}$ 
are separated 
by $2^{-j}$, we have 
\[
\|f\|_{L^{\infty}} \lesssim 2^{jn} \|\alpha\|_{\ell^{\infty}}, 
\quad 
\|g\|_{L^{\infty}} \lesssim 2^{jn} \|\beta\|_{\ell^{\infty}}
\]
(see \cite[(5.13)]{KMT-RMI}). 
Also notice that 
\[
\|h\|_{H^1}= \|F_j\|_{H^1} = \|(\psi)^{\vee}\|_{H^1}<\infty
\]
for all $j$. 
From the choice of $\widetilde{\psi}$, $\theta$, and $\psi$, 
we have 
\[
e^{i|D|} \widetilde{\psi}(2^{-j}D) F_j 
=
e^{i|D|} {\theta}(2^{-j}D) F_j 
=\big( e^{i |\xi|} \psi (2^{-j}\xi) \big)^{\vee} = G_j
\]
and hence 
\begin{align*}
& 
e^{i|D|} \widetilde{\psi}(2^{-j}D) f 
=\sum_{\lambda \in I} 
\alpha_{\lambda} G_j (x-v_{\lambda}), 
\\
& 
e^{i|D|} \widetilde{\psi}(2^{-j}D) g 
=\sum_{\mu \in I} 
\beta_{\mu} G_j (x-v_{\mu}), 
\\
& 
e^{i|D|} \theta (2^{-j}D) h =G_j.
\end{align*}
Thus \eqref{eq521} implies 
\begin{equation}\label{eq522}
\bigg\| 
\sum_{\lambda, \mu \in I}\, 
\alpha_{\lambda} 
\beta_{\mu}
G_j (x-v_{\lambda})
G_j (x-v_{\mu}) 
G_j (x) 
\bigg\|_{L^1_{x}}
\lesssim 
2^{j (-m + 2n)} 
\|\alpha \|_{\ell^{\infty}}
\|\beta \|_{\ell^{\infty}}. 
\end{equation}
We define the cube $Q_{\nu}$ of $\R^n$ by 
\[
Q_{\nu}= v_{\nu} + (0, \delta^{\prime}2^{-j}]^{n}, 
\quad \nu \in I.
\]
Let $\gamma = (\gamma_{\nu})_{\nu \in I}$ be 
a sequence of complex numbers. 
Then, since the cubes $Q_{\nu}$ are mutually disjoint, 
the inequality \eqref{eq522} implies 
\begin{equation}\label{eq523}
\begin{aligned}
&
\bigg| 
\sum_{\lambda, \mu, \nu \in I}\, 
\alpha_{\lambda} 
\beta_{\mu}
\gamma_{\nu}
\int_{Q_{\nu}} 
G_j (x-v_{\lambda})
G_j (x-v_{\mu}) 
G_j (x)\, dx
\bigg|
\\
&
\lesssim 
2^{j (-m + 2n)} 
\|\alpha \|_{\ell^{\infty}}
\|\beta \|_{\ell^{\infty}}
\|\gamma\|_{\ell^{\infty}}. 
\end{aligned}
\end{equation}
Hereafter we shall write 
\[
\int_{Q_{\nu}} 
G_j (x-v_{\lambda})
G_j (x-v_{\mu}) 
G_j (x)\, dx
=
H_j (\lambda, \mu, \nu), 
\quad \lambda, \mu, \nu \in I. 
\]

We now apply Grothendieck's inequality 
(Lemma \ref{lem_Grothendieck}) to the bilinear form 
\[
(\beta, \gamma) \mapsto 
\sum_{\lambda, \mu, \nu \in I}\, 
\alpha_{\lambda} 
\beta_{\mu}
\gamma_{\nu}
H_j (\lambda, \mu, \nu).  
\]
Then from \eqref{eq523} it follows that 
for any Hilbert space $H$ and for any sequences 
$(b_{\mu})_{\mu \in I}$ and 
$(c_{\nu})_{\nu \in I}$ of vectors in $H$  
we have 
\begin{equation*}
\bigg| 
\sum_{\lambda, \mu, \nu \in I}\, 
\alpha_{\lambda} 
\langle b_{\mu}, 
c_{\nu} 
\rangle 
H_j (\lambda, \mu, \nu)
\bigg|
\lesssim 
2^{j (-m + 2n)} 
\|\alpha \|_{\ell^{\infty}}
\sup_{\mu \in I} \|b_{\mu} \|_{H}
\cdot 
\sup_{\nu \in I} \|c_{\nu}\|_{H}. 
\end{equation*}
By duality between $\ell^{\infty}$ and $\ell^{1}$, 
this inequality is equivalent to 
\begin{equation}\label{eq524}
\sum_{\lambda \in I}\, 
\bigg| 
\sum_{\mu, \nu \in I}\, 
\langle b_{\mu}, 
c_{\nu} 
\rangle 
H_j (\lambda, \mu, \nu)
\bigg|
\lesssim 
2^{j (-m + 2n)} 
\sup_{\mu \in I} \|b_{\mu} \|_{H}
\cdot 
\sup_{\nu \in I} \|c_{\nu}\|_{H}. 
\end{equation}

We take $H$, $b=(b_{\mu})$, and $c=(c_{\nu})$ as follows. 
We take $H=L^2 (\R^n)$ with the usual inner product. 
We write 
\begin{align*}
&
V_{j}= \{x \in \R^n \mid 
1-\delta 2^{-j} < |x| < 1+ \delta 2^{-j} 
\}, 
\\
&
V_{j}^{\prime}= \{x \in \R^n \mid 
1-\frac{\delta}{2} 2^{-j} < |x| < 1+ \frac{\delta}{2} 2^{-j}
\}. 
\end{align*} 
For $\mu \in I$, we define 
\begin{align*}
&
E_{\mu}=V_{j}^{\prime} \cap \big(v_{\mu} + V_{j}^{\prime} \big), 
\\
&
b_{\mu} = 2^{j} \ichi_{E_{\mu}} \in H=L^2 (\R^n). 
\end{align*}
We take a sequence $(\epsilon_{\nu})_{\nu \in I}$ consisting 
of $\pm 1$ and define 
\[
c_{\nu} = \epsilon_{\nu} 2^{jn/2} \ichi_{Q_{\nu}} 
\in H=L^2 (\R^n), 
\quad \nu \in I. 
\]
We have 
\begin{equation}\label{eq525}
|E_{\mu}|\approx 2^{-2j} 
\quad \text{for}\quad \mu \in I.
\end{equation} 
This can be seen from the fact that 
for $\mu \in I$ 
the set $\{x \in \R^n \mid |x|=1, \; |x-v_{\mu}|=1\}$ 
is a $(n-2)$-dimensional sphere of radius $\approx 1$ and 
$E_{\mu}$ is nearly equal to the $(\delta /2) 2^{-j}$-neighborhood 
of this set. 
Hence 
\[
\|b_{\mu}\|_{L^2} = 2^{j} |E_{\mu}|^{1/2}\approx 1, 
\quad \mu \in I. 
\]
We also have 
\[
\|c_{\nu}\|_{L^2} = 2^{jn/2} |Q_{\nu}|^{1/2}\approx 1, 
\quad \nu \in I. 
\]  
Thus, with the above $H$, $(b_{\mu})$, and $(c_{\nu})$, 
the inequality \eqref{eq524} gives 

\begin{equation}\label{eq526}
\sum_{\lambda \in I}\, 
\bigg| 
\sum_{\mu, \nu \in I}\, 
\epsilon_{\nu} 
2^{j (1+n/2)} 
\langle \ichi_{E_{\mu}}, 
\ichi_{Q_{\nu}}  
\rangle 
H_j (\lambda, \mu, \nu)
\bigg|
\lesssim 
2^{j (-m + 2n)}.  
\end{equation}

We take the average of \eqref{eq526} over all choices of 
$\epsilon_{\nu}=\pm 1$. 
Then Khintchine's inequality gives 
\begin{equation}\label{eq527}
\sum_{\lambda \in I}\, 
\bigg\| 
\sum_{\mu \in I}\, 
2^{j (1+n/2)} 
\langle \ichi_{E_{\mu}}, 
\ichi_{Q_{\nu}}  
\rangle 
H_j (\lambda, \mu, \nu)
\bigg\|_{\ell^2_{\nu}(I)}
\lesssim 
2^{j (-m + 2n)}.  
\end{equation}

We shall estimate the left hand side of \eqref{eq527} from below. 
For this, observe the following facts. 
\begin{enumerate}
\item $\langle \ichi_{E_{\mu}}, \ichi_{Q_{\nu}} \rangle \ge 0$. 
\item If 
$\langle \ichi_{E_{\mu}}, \ichi_{Q_{\nu}} \rangle \neq 0$, 
then $Q_{\nu} \subset V_{j}$ and $Q_{\nu} - v_{\mu} \subset V_j$ 
(so far as 
$\delta^{\prime}$ is taken sufficiently small compared with $\delta$). 
If in addition $Q_{\nu} - v_{\lambda} \subset V_j$, 
then from Lemma \ref{lem-220919} it follows that 
\[
\RealPart \big( 
e^{-3i \omega_n} 
G_{j}(x-v_{\lambda})
G_{j}(x-v_{\mu})
G_{j}(x)
\big) 
\gtrsim \big( 2^{j(n+1)/2} \big)^{3}
\quad \text{for all}\; \;  
x \in Q_{\nu},  
\]
and hence 
\[
\RealPart \big( 
e^{-3i \omega_n} 
H_j (\lambda, \mu, \nu)
\big) 
\gtrsim \big( 2^{j(n+1)/2} \big)^{3}\,  2^{-jn}, 
\]
and in particular the left hand side of the last inequality is positive. 
\item 
If $Q_{\nu} \subset V_j^{\prime}$, 
then 
\[
\card \{\mu \in I  
\mid 
E_{\mu} \supset Q_{\nu} \}
=
\card \{\mu \in I 
\mid 
v_{\mu} + V_j^{\prime} 
\supset Q_{\nu} \}
\approx 2^{j (n-1)}. 
\]
\item For each $\lambda \in I$, 
\[
\card \{\nu \in I  
\mid 
Q_{\nu}  \subset V_j^{\prime}, 
\; 
Q_{\nu}- v_{\lambda} \subset V_j 
\}
\approx 2^{j(n-2)}
\]
by the same reason as \eqref{eq525}. 
\end{enumerate}

Now using the above facts, 
we shall estimate the left hand side of \eqref{eq527} from below. 
Here we use the letters $\lambda$, $\mu$, $\nu$ to denote 
elements of $I$.  
We have 
\begin{align*}
&
\sum_{\lambda}\, 
\bigg\| 
\sum_{\mu}\, 
2^{j (1+n/2)} 
\langle \ichi_{E_{\mu}}, 
\ichi_{Q_{\nu}}  
\rangle 
H_j (\lambda, \mu, \nu) 
\bigg\|_{\ell^2_{\nu}}
\\
&
\ge 
\sum_{\lambda}\, 
\bigg\| 
\sum_{\mu}\, 
2^{j (1+n/2)} 
\langle \ichi_{E_{\mu}}, 
\ichi_{Q_{\nu}}  
\rangle 
\RealPart \big( e^{-3i \omega_n }  
H_j (\lambda, \mu, \nu) \big) 
\bigg\|_{\ell^2_{\nu}(
Q_{\nu} \subset V_j^{\prime}, \; 
Q_{\nu} - v_{\lambda} \subset V_j
)}
\\
&
\ge 
\sum_{\lambda}\, 
\bigg\| 
\sum_{\mu:\, E_{\mu} \supset Q_{\nu}} 
2^{j (1+n/2)} 
\langle \ichi_{E_{\mu}}, 
\ichi_{Q_{\nu}}  
\rangle 
\RealPart \big( e^{-3i \omega_n }  
H_j (\lambda, \mu, \nu) \big) 
\bigg\|_{\ell^2_{\nu}
( Q_{\nu} \subset V_j^{\prime}, \; 
Q_{\nu} - v_{\lambda} \subset V_j)}
\\
&
\gtrsim 
\sum_{\lambda}\, 
\bigg\| 
\sum_{\mu:\, E_{\mu} \supset Q_{\nu}} 
2^{j (1+n/2)}  2^{-jn} 
\big( 2^{j(n+1)/2}\big)^{3} 2^{-jn} 
\bigg\|_{\ell^2_{\nu}
(Q_{\nu} \subset V_j^{\prime}, \; 
Q_{\nu} - v_{\lambda} \subset V_j)}
\\
&
\approx 
2^{jn} 2^{j(n-2)/2} 2^{j(n-1)} 
2^{j (1+n/2)}  2^{-jn} 
\big( 2^{j(n+1)/2}\big)^{3} 2^{-jn} 
=
2^{j (5n/2 + 1/2)}. 
\end{align*}

Thus \eqref{eq527} implies 
$
2^{j (5n/2 + 1/2)}\lesssim 2^{j(-m+2n)}$, 
which is possible only when $m\le -(n+1)/2$. 
This completes the proof of Proposition \ref{prop_Main2}. 
\end{proof}

\begin{rem}
Here is a comment on the use of 
Grothendieck's inequality in the above proof, 
which the reader may think quite technical.  
In the proof, 
we want to estimate the left-hand side of \eqref{eq523} from below. 
We know that the integral 
$H_{j}(\lambda, \mu, \nu)$ has estimate from below 
if 
\begin{equation}\label{eq241021a}
Q_{\nu} - v_{\lambda} \subset V_j, 
\quad 
Q_{\nu} - v_{\mu} \subset V_j, 
\quad 
Q_{\nu} \subset V_j.
\end{equation} 
The problem is that we cannot simply restrict 
the coefficients 
$\alpha_{\lambda}$, 
$\beta_{\mu}$, 
 $\gamma_{\nu}$ 
so that only the terms satisfying \eqref{eq241021a} appear. 
A way to treat this situation is the following: 
first write \eqref{eq523} in the dual form 
as 
\begin{equation}\label{eq241021b}
\sum_{\lambda \in I}\, 
\bigg| 
\sum_{\mu, \nu \in I}\, 
\beta_{\mu} 
\gamma_{\nu} 
H_j (\lambda, \mu, \nu)
\bigg|
\lesssim 
2^{j (-m + 2n)} 
\|\beta \|_{\ell^{\infty}}
\|\gamma\|_{\ell^{\infty}}, 
\end{equation}
and then setting $\beta_{\mu}= \pm 1$, 
$\gamma_{\nu}=\pm 1$, 
take average over $\pm 1$ 
and use Khintchine's inequality twice. 
This procedure gives the condition $m \le -n/2$. 
If we use 
Grothendieck's inequality and use 
\eqref{eq524} instead of \eqref{eq241021b}, 
then we have more freedom to choose $b_{\mu}, c_{\nu}\in H$, 
in fact we can have $\langle b_{\mu}, c_{\nu}\rangle =0$ 
with $b_{\mu}, c_{\nu}\neq 0$. 
Our choice of $H$, $b_{\mu}$, and $c_{\nu}$ 
as given in the proof of Proposition \ref{prop_Main2} 
enabled us to reduce the use of Khintchine's inequality 
and gave stronger condition $m \le -(n+1)/2$. 
\end{rem}



\begin{thebibliography}{200000}

\bibitem[BRRS]{BRRS-Trans}
A. Bergfeldt, S. Rodr\'iguez-L\'opez, D. Rule, and W. Staubach,
Multilinear oscillatory integrals and estimates for coupled systems of dispersive PDEs,
Trans. Amer. Math. Soc. {\bf 376} (2023), 7555--7601.

\bibitem[CM1]{CM1}
R. Coifman and Y. Meyer,
{Au del\`a des op\'erateurs pseudo-diff\'erentiels},
Ast\'erisque {\bf 57} (1978), 1--199.


\bibitem[Ga]{Garling}
D. J. H. Garling,
{\it Inequalities:  
A Journey into Linear Analysis\/}, 
Cambridge Univ. Press, 2007.

\bibitem[Go]{Goldberg}
D. Goldberg, 
A local version of real Hardy spaces, 
Duke Math. J. 
{\bf 46} (1979), 27--42. 



\bibitem[Gr2]{G-modern}
L. Grafakos,
{\it Modern Fourier Analysis\/}, 
3rd edition, Graduate Texts in Math. 250, Springer, 2014.

\bibitem[GrP]{Grafakos-Peloso}
L. Grafakos and M. M. Peloso, 
Bilinear Fourier integral operators, 
J. Pseudo-Differ. Oper. Appl. (2010) 1:161--182, 
DOI 10.1007/s11868-010-0011-4


\bibitem[KMST]{KMST2}
T. Kato, A. Miyachi, N. Shida, and N. Tomita, 
{On some bilinear Fourier multipliers with 
oscillating factors, II}, 
in preparation. 

\bibitem[KMT1]{KMT-RMI}
T. Kato, A. Miyachi, and N. Tomita, 
{Estimates for some bilinear wave operators}, 
Rev. Mat. Iberoam. 
{\bf 40} (2024), 1571--1608.

\bibitem[KMT2]{KMT-JPOA}
T. Kato, A. Miyachi, and N. Tomita, 
{Estimates for a certain bilinear Fourier integral operators}, 
J. Pseudo-Differ. Oper. Appl. (2024) 15:59, 
https://doi.org/10.1007/s11868-024-00631-0

\bibitem[M]{M-wave}
A. Miyachi, 
On some estimates 
for the wave equations in $L^p$ and in $H^p$, 
J.\ Fac.\ Sci. Univ.\ Tokyo, Sect.\ IA Math.\ {\bf 27} 
(1980), 331--354.

\bibitem[P]{Peral}
J. C. Peral, 
$L^p$ estimates for the wave equation, 
J. Functional Analysis 
\textbf{36} (1980), no. 1, 114--145.
MR568979 (81k:35089) 

\bibitem[RRS1]{RRS1}
S. Rodr\'iguez-L\'opez, D. Rule, and W. Staubach,
A Seeger--Sogge--Stein theorem for bilinear Fourier 
integral operators, 
Adv. Math. {\bf 264} (2014), 1--54. 


\bibitem[RRS2]{RRS2}
S. Rodr\'iguez-L\'opez, D. Rule, and W. Staubach,
Global boundedness of a class 
of multilinear Fourier 
integral operators,
Forum Math., Sigma {\bf 9:e14} (2021), 1--45.

\bibitem[SSS]{SSS}
A. Seeger, C. D. Sogge, and E. M. Stein, 
Regularity properties of Fourier integral operators, 
Ann. of Math. {\bf 134}(2) (1991), 231--251. 


\bibitem[Sj]{Sj}
S. Sj\"ostrand, 
On the Riesz means of the solutions of the 
Schr\"odinger equation,  
Ann. Scuola Norm. Sup. Pisa {\bf 24} (1970), 331--348. 


\bibitem[St]{St}
E. M. Stein, 
{\it Harmonic Analysis:  
Real-Variable Methods, 
Orthogonality, 
and Oscillatory Integrals}, 
Princeton Univ.\ Press, 
Princeton, NJ, 1993.
\end{thebibliography}
\end{document}